\documentclass[12pt]{amsart}
\usepackage{a4wide}
\usepackage{amssymb,bm,mathrsfs}
\usepackage{amsthm}
\usepackage{amsmath}
\usepackage{amscd}
\usepackage{verbatim}
\usepackage[all]{xy}

\usepackage{hyperref}

\numberwithin{equation}{section}

\theoremstyle{plain}
\newtheorem{theorem}{Theorem}[section]
\newtheorem{corollary}[theorem]{Corollary}
\newtheorem{lemma}[theorem]{Lemma}
\newtheorem{proposition}[theorem]{Proposition}

\theoremstyle{definition}
\newtheorem{definition}[theorem]{Definition}
\newtheorem{example}[theorem]{Example}

\theoremstyle{remark}
\newtheorem{remark}[theorem]{Remark}

\newcommand{\R}{\mathbb{R}}

\newcommand{\Q}{\mathbb{Q}}
\newcommand{\Z}{\mathbb{Z}}
\newcommand{\N}{\mathbb{N}}
\newcommand{\C}{\mathbb{C}}
\renewcommand{\H}{\mathbb{H}}

\newcommand{\F}{\mathbb{F}}

\newcommand{\zxz}[4]{\begin{pmatrix} #1 & #2 \\ #3 & #4 \end{pmatrix}}
\newcommand{\abcd}{\zxz{a}{b}{c}{d}}

\newcommand{\kzxz}[4]{\left(\begin{smallmatrix} #1 & #2 \\ #3 & #4\end{smallmatrix}\right) }
\newcommand{\kabcd}{\kzxz{a}{b}{c}{d}}

\renewcommand{\Im}{\operatorname{Im}}

\newcommand{\calA}{\mathcal{A}}

\newcommand{\calD}{\mathcal{D}}
\newcommand{\calE}{\mathcal{E}}
\newcommand{\calF}{\mathcal{F}}
\newcommand{\calG}{\mathcal{G}}

\newcommand{\calI}{\mathcal{I}}

\newcommand{\calO}{\mathcal{O}}

\newcommand{\calR}{\mathcal{R}}
\newcommand{\calS}{\mathcal{S}}

\newcommand{\eps}{\varepsilon}
\newcommand{\bs}{\backslash}

\newcommand{\tr}{\operatorname{tr}}

\newcommand{\Sl}{\operatorname{SL}}
\newcommand{\Gl}{\operatorname{GL}}
\newcommand{\SL}{\operatorname{SL}}
\newcommand{\Symp}{\operatorname{Sp}}

\newcommand{\Uni}{\operatorname{U}}

\newcommand{\Aut}{\operatorname{Aut}}
\newcommand{\Mat}{\operatorname{Mat}}
\newcommand{\Spec}{\operatorname{Spec}}

\newcommand{\acn}{\operatorname{acn}}

\newcommand{\dv}{\operatorname{div}}

\newcommand{\Sym}{\operatorname{Sym}}

\newcommand{\GL}{\operatorname{GL}}

\newcommand{\tor}{\operatorname{tor}}

\begin{document}

\title[Formal Siegel modular forms]{Formal Siegel modular forms for arithmetic subgroups}

\author[Jan H.~Bruinier and Martin Raum]{Jan Hendrik Bruinier and Martin Raum}
\address{Technische Universit\"at Darmstadt, Fachbereich Mathematik, Schlossgartenstraße 7,
D--64289 Darmstadt, Germany}
\email{bruinier@mathematik.tu-darmstadt.de}

\address{
Chalmers tekniska högskola och Göteborgs Universitet, Institutionen för
Matematiska vetenskaper, SE--412 96 Göteborg, Sweden
}
\email{martin@raum-brothers.eu}


\thanks{The first author is supported in part by  the
DFG Collaborative Research Centre TRR 326 ``Geometry and Arithmetic of Uniformized Structures'', project number 444845124. The second author was partially supported by Vetenskapsrådet grant 2019-03551 and
2023-04217.
}

%\date{\today}

\begin{abstract}
%A formal Siegel modular form of cogenus $l$  for an arithmetic subgroup $\Gamma$ of the symplectic group of genus $n$ is given by a collection of formal Fourier-Jacobi series for the rational boundary components of degree $n-l$ that satisfy a certain compatibility condition.
The notion of formal Siegel modular forms for an arithmetic subgroup $\Gamma$ of the symplectic group of genus $n$ is a generalization of symmetric formal Fourier-Jacobi series. 
Assuming an upper bound on the affine covering number of the Siegel modular variety associated with $\Gamma$, we prove that all formal Siegel modular forms are given by Fourier-Jacobi expansions of classical holomorphic Siegel modular forms. 
%The proof is based on interpreting formal Siegel modular forms as global sections of the line bundle of modular forms on the formal completion of the Satake compactification at its boundary and on an algebraicity theorem of Raynaud.
We also show that the required upper bound is always met if $2\leq n \leq 4$.
As an application we consider the case of the paramodular group of squarefree level and genus $2$. 
\end{abstract}

\maketitle

%\tableofcontents

%%%%%%%%%%%%%%%%%%%%%%%%%%%%%%%%%%%%%%%%%%%%%%%%%%%%%%%%

\section{Introduction}

%%%%%%%%%%%%%%%%%%%%%%%%%%%%%%%%%%%%%%%%%%%%%%%%%%%%%%%

Let $f$ be a holomorphic Siegel modular form  of weight $k$ for the full Siegel modular group $\Gamma_n=\Symp_n(\Z)$ of genus $n$, that is, a holomorphic function on the Siegel upper half space $\H_n$ which transforms in weight $k$ under the action of $\Symp_n(\Z)$. For any integer $l$ with $0\leq l \leq n$, the function $f$ has a Fourier-Jacobi expansion of cogenus $l$ of the form 
\begin{align}
\label{eq:intro1}
f(\tau) = \sum_{T_2} \phi_{T_2}(\tau_1,\tau_{12}) e(\tr(T_2 \tau_2)),
\end{align}
where we have written the variable $\tau \in \H_n$ in block form as 
\[
\tau=\zxz{\tau_1}{\tau_{12}}{{}^t\tau_{12}}{\tau_2}
\]
with $\tau_1\in \H_{n-l}$, $\tau_2\in \H_{l}$, and $\tau_{12}\in \C^{(n-l)\times l}$. Moreover, $T_2$ runs over all positive semidefinite symmetric half-integral $l\times l$ matrices.
The series \eqref{eq:intro1} converges normally on $\H_n$. In the special case $l=n$ it reduces to the usual Fourier expansion of $f$.

The transformation law of $f$ under the stabilizer 
%$\Gamma_{n,n-l}\subset 
in $\Symp_n(\Z)$ of the standard boundary component of degree $n-l$ implies that the coefficients $\phi_{T_2}$ are Jacobi forms of weight $k$, index $T_2$, and genus $n-l$. In particular, they possess Fourier expansions of the form 
\begin{align}
\label{eq:intro2}
\phi_{T_2}(\tau_1,\tau_{12}) = 
%\sum_{\substack{T_1\in \Sym_{n-l}(\Q)\\T_{12} \in \Q^{(n-l)\times l}}}
\sum_{T_1,T_{12}} a\zxz{T_1}{T_{12}}{{}^tT_{12}}{T_2} e\left(\tr(T_1 \tau_1)+ 2\tr(T_{12}{}^t\tau_{12})\right).
\end{align}
Here $T_1$ runs through positive semidefinite symmetric half-integral $(n-l)\times (n- l)$ matrices, and $T_{12}$ runs through $(n-l)\times l$ matrices with half-integral entries. Inserting these expansions into \eqref{eq:intro1}, one recovers the usual Fourier expansion of $f$, which explains our particular notation for the Fourier coefficients in  \eqref{eq:intro2}.
The transformation law of $f$ under the Siegel parabolic subgroup (the stabilizer 
%$\Gamma_{n,n-l}\subset in $\Symp_n(\Z)$ 
of the standard boundary component of degree $0$) implies that the Fourier coefficients have the symmetry property
\begin{align}
\label{eq:intro3}
a( {}^t u T u) = \det(u)^k a(T)
\end{align}
for all $T= \kzxz{T_1}{T_{12}}{{}^tT_{12}}{T_2}$ and all $u\in \GL_n(\Z)$. 

In recent works, several authors considered formal analogues of Fourier-Jacobi expansions. A {\em formal} Fourier-Jacobi series $f$ of weight $k$ and cogenus $l$ for the group $\Symp_n(\Z)$ is a formal series as in \eqref{eq:intro1}, where the coefficients $\phi_{T_2}$ are holomorphic Jacobi forms of weight $k$, index $T_2$, and genus $n-l$, but where no convergence of the series is required. By considering the Fourier expansions of the $\phi_{T_2}$ as in \eqref{eq:intro2} and inserting them into the series \eqref{eq:intro1}, one obtains a formal Fourier expansion of $f$. Recall that $f$ is called {\em symmetric}, if the Fourier coefficients satisfy the symmetry condition \eqref{eq:intro3} for all half-integral symmetric metrices $T$ and all $u\in \GL_n(\Z)$.

In particular, every holomorphic Siegel modular form of weight $k$ for $\Symp_n(\Z)$ defines a corresponding symmetric formal Fourier-Jacobi series of cogenus $l$. The main result of \cite{BR} states that for $n>1$ and for  every $l$ with $0<l<n$, the converse also holds: Every symmetric  formal Fourier-Jacobi series of weight $k$ and cogenus $l$ for $\Symp_n(\Z)$ arises as the Fourier-Jacobi expansion of a classical holmorphic Siegel modular form.

In the special case when $n=2$ this result was first proved by Aoki in \cite{Ao}. It was generalized to vector valued symmetric formal Fourier-Jacobi series for $\Symp_2(\Z)$ in \cite{Br1} and \cite{Rau}. The case of the paramodular group of genus $2$ and level $\leq 4$ was considered in \cite{IPY}. 
Part of the argument of \cite{BR} is revisited in \cite{Kr} using the arithmetic theory of Siegel modular forms of Faltings-Chai. Recent work of Xia  deals with the somewhat analogous case of the unitary group $\Uni(n,n)$ over a norm-euclidian imaginary quadratic field \cite{Xia}. 
Note that all these works deal with (vector valued) formal Fourier-Jacobi series for the full Siegel modular group of level $1$. The proofs mainly rely on analytic techniques, such as bounds on the dimension of the space of symmetric formal Fourier-Jacobi series 
as the weight  goes to infinity.

The purpose of the present paper is two-fold. 
First, we present a new approach to the problem. It is based on interpreting symmetric formal Fourier-Jacobi series as global sections of the line bundle of modular forms on the formal complex space given by the  completion of the Satake compactification at its boundary. Second, we generalize the notion of  symmetric formal Fourier-Jacobi series  to arithmetic subgroups of the Siegel modular group  and  prove modularity results in this context. We now describe our results in more detail.

Let $\Gamma\subset \Symp_n(\Q)$ be an arithmetic subgroup, that is, a subgroup which is commensurable with $\Gamma_n=\Symp_n(\Z)$. Recall that the rational closure $\H_n^*$ of $\H_n$ is the disjoint union of $\H_n$ with all its proper rational boundary components. If $\H_n^*$ is equipped with the cylindrical topology, then  
$\Gamma$ acts properly discontinuously on it, extending the action on $\H_n$ by fractional linear transformations. 
The Satake compactification of the Siegel modular variety $X_\Gamma=\Gamma\bs \H_n$ is given by the quotient $X_\Gamma^*=\Gamma_n\bs \H_n^*$, equipped with the Satake complex structure. It is a normal complex space, which has a natural structure as a projective algebraic variety over $\C$. The Satake boundary  is a closed subspace of codimension $n$.

For $0\leq l\leq n$, let $I_l$ be the set of all rational boundary components of degree $l$.
We define a {\em formal Siegel modular form} of weight $k$ and cogenus $l$
for the group  $\Gamma$ to be a family $(f_F)_{F\in I_{n-l}}$, where $f_F$ is a  formal Fourier-Jacobi series of weight $k$ for the boundary component $F$ and the group $\Gamma$ satisfying the following conditions:
\begin{itemize}
\item[(i)] for all $F\in I_{n-l}$ and all $\gamma \in \Gamma$ we have 
$f_F\mid_k \gamma= f_{\gamma^{-1} F}$;
\item[(ii)] for all pairs $F,F'\in I_{n-l}$ and all degree $0$ boundary components $E\in I_0$ which are adjacent to  $F$ and $F'$, the series $f_F$ and $f_{F'}$ are compatible at $E$.
\end{itemize}
Here, the action of $\Gamma$ on $f_F$ in (i) is defined by the action on the coefficients of the formal Fourier-Jacobi series. The compatibility condition in (ii) means that $f_F$ and $f_{F'}$ have the same formal Fourier expansion at $E$, see Section \ref{sect:3.1} and Definition~\ref{def:symffj} for details.
We write $
\operatorname{FM}^{(n,l)}_k(\Gamma)$ for the complex vector space of formal Siegel modular forms of weight $k$ and cogenus $l$ for $\Gamma$.

In the special case when $\Gamma=\Gamma_n$ is the full integral symplectic group and $F_{n-l}$ is the standard rational boundary component of degree $n-l$, a formal Fourier-Jacobi series of weight $k$ for the boundary component $F_{n-l}$ and the group $\Gamma_n$ is just a formal Fourier-Jacobi series of cogenus $l$ as considered before. Using the fact that $\Gamma_n$ acts transitively on $I_{n-l}$, it is easily seen that the compatibility condition (ii) implies the symmetry condition \eqref{eq:intro3}. Hence formal Siegel modular forms of weight $k$ and cogenus $l$ for $\Gamma_n$ can be identified with symmetric formal Fourier-Jacobi series of the same type.

In Section \ref{sect:3.3} we give an algebraic geometric description of formal Siegel modular forms of cogenus $1$.
Let $\omega$ be the sheaf of modular forms of weight $1$ on $X_\Gamma^*$.
Let $\hat X^*_\Gamma$ be the formal complex space given by the completion of $X^*_\Gamma$ at the Satake boundary $Y=X^*_\Gamma\setminus X_\Gamma$, and write 
\begin{align*}
i:\hat X_\Gamma^*\to X_\Gamma^*
\end{align*}
for the natural morphism of formal complex spaces. 
We denote by $\hat\omega^{\otimes k}$ the completion of the sheaf $\omega^{\otimes k}$ of modular forms of weight $k$ with respect to $Y$. Since $\omega^{\otimes k}$ is coherent, the natural map $i^*( \omega^{\otimes k})\to \hat\omega^{\otimes k}$ is an isomorphism. We provide an explicit description of the sections of $\hat\omega^{\otimes k}$ over a small open neighborhood of any rational boundary component $F$ in Proposition \ref{prop:compl}. To this end we use the Grothendieck comparison theorem 
% \cite[Theorem 4.1.5]{EGA} 
in the category of formal complex analytic spaces \cite[Theorem 2]{Ba} to reduce the computation to a smooth toroidal compactification. In particular, we find that for any rational boundary component $F$ of degree $n-1$, the sections of $\hat\omega^{\otimes k}$ over a small open neighborhood of $F$ can be identified with formal Fourier-Jacobi series of weight $k$ for $F$. Moreover, these series satisfy the compatibility condition (ii) above. Hence, we obtain an injective linear map
\begin{align}
\label{eq:intro4}
\hat\omega^{\otimes k}(\hat X_\Gamma^*)\to \operatorname{FM}^{(n,1)}_k(\Gamma)
\end{align}
taking a global section to its formal Fourier-Jacobi expansions of cogenus $1$.
Our first result is a follows (see Theorem \ref{prop:ffj}).

\begin{theorem}
\label{thm:intro1}
The map 
%$\hat\omega^{\otimes k}(\hat X^*)\to \operatorname{FM}^{(n,1)}_k(\Gamma)$ 
in  \eqref{eq:intro4} is an isomorphism.
\end{theorem}

Recall that the {\em affine covering number} $\acn(S)$ of  a scheme $S$ is defined as one less than the smallest number of open affine sets required to cover $S$, see e.g.~\cite{At}, \cite{RV}. It gives an upper bound for the cohomological dimension of $S$.
%, which is the largest integer $j$ such that $H^j(S, \calF ) \neq  0$ for some quasicoherent sheaf $\calF$ \cite[Proposition 4.12]{RV}. 
If $S$ is a quasi-projective scheme, then there is the 
trivial bound  $\acn(S)\leq \dim(S)$. Using this notion, we may formulate our main result (see Theorem \ref{thm:maingen}).

\begin{theorem}
\label{thm:intromaingen}
Assume that $\acn(X_\Gamma)\leq \frac{n(n+1)}{2} -2$. Then the natural map 
\begin{align}
\label{eq:intro5}
H^0(X^*_\Gamma, \omega^{\otimes k})\to  \operatorname{FM}^{(n,1)}_k(\Gamma)
\end{align}
taking a holomorphic modular form to its cogenus $1$ formal Fourier-Jacobi expansions is an isomorphism. 
\end{theorem}

To prove this result, we use  an algebraization theorem of Raynaud \cite[Corollaire 2.8]{Ray} to show that the  natural map  $H^0(X^*_\Gamma, \omega^{\otimes k})\to H^0(\hat X^*_\Gamma, \hat \omega^{\otimes k})$ is an isomorphism. Then the assertion can be deduced by means of Theorem \ref{thm:intro1},  see Section \ref{sect:4.2}.

This raises the problem of computing (or bounding) the affine covering number of the Siegel modular variety $X_\Gamma$. 
By the theory of Baily-Borel, $\acn(X_\Gamma)$ is the smallest non-negative integer $j$, for which there exist cusp forms $F_0,\dots, F_j$ for $\Gamma$ 
that have no common zero on $\H_n$.

According to \cite[Theorem 4]{At}, we have $\acn(X_\Gamma)\geq n(n-1)/2$. 
It is 
believed that this lower bound is actually an equality.
However, for general $n$ not much is known in this direction. Here we employ results of Igusa, Salvati Manni, and  Fontanari--Pascolutti on the reducible locus of $X_{\Gamma_n}$ to prove the required upper bounds for small $n$, see Proposition~\ref{prop:acn3} and Proposition~\ref{prop:acn4}.
The following corollaries can be derived.

\begin{corollary}
\label{cor:intromaingen}
Assume that $2\leq n \leq 4$. Then the natural map \eqref{eq:intro5}
%\[
%H^0(X^*_\Gamma, \omega^{\otimes k})\to  \operatorname{FM}^{(n,1)}_k(\Gamma)
%\]
%taking a modular form to its cogenus $1$ formal Fourier-Jacobi expansion 
is an isomorphism. 
\end{corollary}

\begin{corollary}
Assume that $2\leq n \leq 4$. Let $U\subset X_\Gamma^*$ be an open analytic neighborhood of the Satake boundary $Y$.
Then the restriction map $H^0(X^*_\Gamma,\omega^{\otimes k})\to H^0(U,\omega^{\otimes k})$ is an isomorphism.
\end{corollary}

We remark that Theorem \ref{thm:intromaingen} and Corollary \ref{cor:intromaingen} have natural generalizations to vector valued modular forms transforming with a finite dimensional representation of $\Gamma$. Alternatively, one can derive such results for vector valued forms from the scalar case by means of the argument of \cite{Br1}.
Using induction on the cogenus as in \cite[Lemma 5.2]{BR} one can also deduce an analogue for formal Siegel modular forms of higher cogenus $l<n$.

As an application, we consider the case of the paramodular group $K(N)\subset \Symp_2(\Q)$ of level $N$ and genus $2$, see Section \ref{sect:4.4}. We write $K(N)^*$ for the extension of $K(N)$ by all Atkin-Lehner type involutions.
It contains $K(N)$ as a normal sugbroup, and 
\[
K(N)^*/K(N) \cong (\Z/2\Z)^{\nu(N)},
\]
where $\nu(N)$ denotes the number of positive divisors of $N$.
Let $f$ be a formal Fourier-Jacobi series of weight $k$ for the standard boundary component $F_1$ and the group $K(N)$. Denote by 
\begin{align*}
f(\tau) = \sum_{T} a(T)  \,e(\tr T \tau)
\end{align*} 
its formal Fourier expansion at the boundary component $F_0$. Here $T$ runs through all half integral positive semi-definite matrices $2\times 2$-matrices with  $N\mid T_2$.
% as in \eqref{eq:fouf}.
We call $f$  {\em strongly symmetric} if there exists a character $\chi_f: K(N)^*/K(N)\to \{\pm 1\}$ such that 
\begin{align}
\label{intro:strongsymm}
a(uT\, {}^t u) = \chi_f\zxz{{}^tu^{-1}}{0}{0}{u}\det(u)^k a(T) 
\end{align} 
for all $T$ and all $u\in \Gamma_0(N)^*$, the extension of $\Gamma_0(N)$  by the Atkin-Lehner involutions (viewed as elements of $\SL_2(\R)$).

\begin{theorem}
\label{thm:introKN}
Let $N$ be a square-free positive integer.
Let $f$ be a strongly symmetric formal Fourier-Jacobi series of weight $k$ for the boundary component $F_1$ and the group $K(N)$. Then $f$ converges and defines a 
paramodular form in $M_k(K(N))$.
\end{theorem}

For a discussion of the relationship of the symmetry condition in this result and  the involution condition in the work of Ibukiyama,  Poor, and Yuen \cite{IPY} we refer to Section~\ref{sect:4.4}.

Part of the motivation to investigate the modularity of formal Fourier-Jacobi series comes from the Kudla program. Here certain generating series of classes of special cycles in Chow groups of orthogonal and unitary Shimura varieties play a central role, see e.g.~\cite{Ku:Duke}, \cite{Ku:Annals}, \cite{Ku:MSRI}. The generating series of special cycles of codimension $n$ is conjectured to be a Siegel (respectively Hermitian) modular form  of genus $n$. By an argument of Wei Zhang \cite{Zh} one can often show that the generating series is given by a symmetric formal Fourier-Jacobi series of genus $n$ and cogenus $n-1$. Hence the conjectured modularity can be deduced from a suitable modularity result for symmetric formal Fourier-Jacobi series. In this way, Kudla's modularity conjecture was established for orthogonal Shimura varieties associated with quadratic spaces of signature $(m,2)$ over $\Q$ in \cite{BR} and for unitary Shimura varieties associated with hermitian spaces of signature $(m,1)$ over norm-euclidian imaginary quadratic fields in \cite{Xia}, based on Liu's extension of Zhang's
work~\cite{Liu}. An analogous result for special cycles on integral models of orthogonal Shimura varieties is proved in \cite{HM}. Symmetric formal Fourier-Jacobi series can also be used for the computation of Siegel modular forms, see e.g.~\cite{IPY}. 

We briefly describe the contents of this paper. In Section~\ref{sect:2} we recall some facts on Siegel modular varieties, the Satake compactification, and on Siegel modular forms. This mainly serves to fix  notation and to collect some important facts that will be used later.
In Section~\ref{sect:3} we introduce formal Siegel modular forms and provide the algebraic geometric description as formal sections of the line bundle of Siegel modular forms. In Section~\ref{sect:4} we consider affine covering numbers of Siegel modular varieties, Raynaud's algebraization theorem, and its application to formal Siegel modular forms. Finally we discuss the case of the paramodular group in genus $2$.

\subsubsection*{Acknowledgments}
We thank Jean-Beno\^it Bost, Eberhard Freitag, Samuel Grushevsky, Klaus Hulek, J\"urg Kramer, Riccardo Salvati Manni, and Torsten Wedhorn for helpful conversations on the content of this paper.

%%%%%%%%%%%%%%%%%%%%%%%%%%%%%%%%%%

\section{Siegel Modular varieties}
\label{sect:2}
%%%%%%%%%%%%%%%%%%%%%%%%%%%%%%%%%%

Here we recall some facts on Siegel modular varieties, the Satake compactification, and Siegel modular forms. This mainly serves to fix notation and to provide some background for the following sections.

Let $n$ be a positive integer, and denote by 
$W=\Q^{2n}$ the standard symplectic space of dimension $2n$, equipped with the symplectic form given by  
$\langle x,y\rangle = x J\, {}^t y$, where 
\[J=\zxz{0}{1_n}{-1_n}{0}
\]
($x,y\in W$ are viewed as row vectors). 
%Let 
%$\underline G = \Symp_n$ be the symplectic group of $W$, viewed as an algebraic group over $\Q$. Its  
%$R$-points are given by 
%\[
%\Symp_n(R)= \{ g\in \GL_{2n}(R)\mid \; g J \,{}^tg= J\}
%\]
%for every $\Q$-algebra $R$.
%the real symplectic group of genus $n$, where 
%\[J=\zxz{0}{1_n}{-1_n}{0}
%\]
%denotes the standard symplectic form $\langle x,y\rangle = x J\, {}^t y$ on $\R^{2n}$ (where $x,y$ are viewed as row vectors). 
Write $\Sym_n(\C)$ for the space of symmetric complex $n\times n$-matrices.
The real symplectic group $G:=\Symp_n(\R)$ acts on the Siegel upper half space
$\H_n= \{\tau\in \Sym_n(\C)\mid \; \Im(\tau)>0\}$ 
by fractional linear transformations
\[
\tau \mapsto \abcd\tau = (a\tau +b)(c\tau+d)^{-1}.
\]
The Cayley transformation $\tau\mapsto z=(\tau-i1_n)(\tau+i1_n)^{-1}$ maps $\H_n$ biholomorphically to the bounded symmetric domain 
\[
\calD_n = \{z \in \Sym_n(\C)\mid \;  1_n-z\bar z >0\}.
\]
The action of $G$ on $\H_n$ induces a compatible action on $\calD_n$.

\subsection{Boundary components}

The action of $G$ extends to the topological closure $\bar \calD_n$ of $\calD_n$ in $\Sym_n(\C)$. 
Two points in $\bar \calD_n$ are called equivalent if they can be connected by a finite chain of holomorphic curves $\xi_i: \{z\in \C\mid \; |z|<1\} \to \bar \calD_n$.
It is easily seen that all points in $\calD_n$ are equivalent. The equivalence classes in 
$\bar\calD_n\setminus \calD_n$ are called the proper {\em boundary components} of $\calD_n$.

To any $z\in \bar\calD_n$ we can associate a linear map 
\[
\psi_z: \R^{2n}\to \C^n,\quad \nu \mapsto \nu \begin{pmatrix} i(1_n+z) \\1_n-z\end{pmatrix},
\]
where the elements of $\R^{2n}$ and $\C^n$ are viewed as row vectors. The subspace $U(z)= \ker\psi_z \subset \R^{2n}=W_\R$ is totally isotropic with respect to the symplectic form $J$. It is non-trivial if and only if $z\in \bar\calD_n\setminus \calD_n$ is a proper boundary point. Moreover, $U(z_1)=U(z_2)$ if and only if $z_1$ and $z_2$ are equivalent.
Hence, we obtain a bijection $F\mapsto U(F)$ between the set of proper boundary components $F$ of $\bar \calD_n$ and the set of non-trivial isotropic subspaces $U\subset \R^{2n}$, see \cite[Chapter I.3A]{HKW} and \cite[Section 4]{Na}.
The group $G$ acts on isotropic subspaces $U\subset \R^{2n}$ by right translation $U\mapsto Ug^{-1}$ for $g\in G$. This action is compatible with the action on $\bar\calD_n$, as we have $U(gz)= U(z)g^{-1}$.

Recall that a boundary component $F$ is called {\em adjacent} to another boundary component $F'$,  if $\bar F'\supset F$ and $F'\neq F$. In this case we write $F'>F$. This is equivalent to the condition that the isotropic subspace $U(F)$ correponding to $F$ strictly contains the subspace $U(F')$.

For $0\leq m\leq n$, the subset
\[
F_{m} =\left\{ \zxz{z'}{0}{0}{1_{n-m}}\mid \; z'\in \calD_{m}\right\} \subset \bar\calD_n
\]
is a boundary component of $\calD_n$, called the {\em standard boundary component} of degree $m$. The corresponding isotropic subspace $U(F_{m}) $ has dimension $n-m$ and is given by 
\begin{align}
\label{eq:UFm}
U(F_{m}) = \{ (x_1,\dots,x_{2n})\in \R^{2n}\mid \; x_1=\dots= x_{n+m}=0\}.
\end{align}
% Could call this U_{m}.
In particular, we have $F_n=\calD_n$, and $F_0$ is a point. The Caley transformation $\H_{m}\to \calD_{m}$ induces an isomorphism $\H_{m}\to F_{m}$.

The stabilizer 
\begin{align}
G_F=\{g\in G\mid\; \text{$g(F)=F$}\} 
\end{align}
of a boundary component $F$ is a maximal parabolic subgroup of $G$. 
We also consider the centralizer 
 \begin{align}
G_F^0= \{g\in G\mid\; \text{$g(z)=z$ for all $z\in F$}\} 
\end{align}
of $F$, which is a normal subgroup of $G_F$. We denote by $G_F'$ the center of the unipotent radical of $G_F$, and by 
\[
G^J_F=\{g\in G\mid \; \text{$ghg^{-1} = h$ for all $h\in  G_F'$}\}
\]
its centralizer  in the group $G$. 
%
%We denote the quotient by $\bar G_F=G_F/G_F^0$.
Recall from \cite[\S 7]{Na} and \cite[Chapter III.4]{AMRT} that here is a homomorphism 
\begin{align}
\label{eq:pl}
p_\ell: G_F\to \Aut(G_F'),\quad g\mapsto p_\ell(g)= (h\mapsto ghg^{-1}). 
\end{align}
The image of $p_\ell$ preserves the quadratic form on $G_F'$ induced by the Killing form and the cone of positive elements. Moreover, we have $G_F^J=\ker(p_\ell)$. 

For the standard boundary component $F_{m}$ we have 
\begin{align}
\label{eq:GFm}
G_{F_{m}}&=\left\{ \begin{pmatrix}
a & 0 & b & *\\
* & u & * & *\\
c & 0 & d & *\\
0 & 0 & 0 & {}^tu^{-1}
\end{pmatrix}\right\},\\
G_{F_{m}}^0&=\left\{ \begin{pmatrix}
1 & 0 & 0 & *\\
* & u & * & *\\
0 & 0 & 1 & *\\
0 & 0 & 0 & {}^tu^{-1}
\end{pmatrix}\right\},
\end{align}
where $\kabcd\in \Symp_{m}(\R)$ and $u\in \Gl_{n-m}(\R)$. The quotient group $G_{F_{m}}/G_{F_{m}}^0$ is isomorphic to $\Symp_{m}(\R)$.
Moreover, it is easily seen that 
\begin{align*}
G'_{F_{m}}&=\left\{ \begin{pmatrix}
1 & 0 & 0 & 0\\
0 & 1 & 0 & s\\
0 & 0 & 1 & 0\\
0 & 0 & 0 & 1
\end{pmatrix}\right\},
%,\\
%G_{F_{m}}^J&=\left\{ \begin{pmatrix}
%a & 0 & b & *\\
%* & \pm 1 & * & *\\
%c & 0 & d & *\\
%0 & 0 & 0 & \pm 1
%\end{pmatrix}\right\},
\end{align*}
and the normalizer of $G'_{F_{m}}$ in $G$ is given by $G_{F_m}$. The group 
$G^J_F$ consists of those matrices in $G_{F_m}$ for which $u=\pm 1_{n-m}$.
If we identify $G_{F_m}'\cong \Sym_{n-m}(\R)$, then the action of $g\in G_{F_m}$ as in \eqref{eq:GFm} on $s\in \Sym_{n-m}(\R)$ is given by $p_\ell(g)(s)= s[{}^tu]=u s\,{}^tu$.
Let $\Gamma_n=\Symp_n(\Z)$ and put 
\begin{align}
\label{eq:gnm}
\Gamma_{n,m} &= \Gamma_n\cap G_{F_{m}},\\ 
\label{eq:gnm0}
\Gamma_{n,m}^0 &= \Gamma_n\cap G_{F_{m}}^0.
%\\
%\Gamma_{n,m}' &= \Gamma_n\cap G_{F_{m}}',\\
%\Gamma_{n,m}^J &= \Gamma_n\cap G_{F_{m}}^J.
\end{align}

A boundary component $F$ is called {\em rational} if $G_F$ is defined over $\Q$. 
%(with respect to the natural rational structure on $G$).
This is equivalent to the condition that $U(F)\subset W_\R$ is defined over $\Q$. Moreover, it is equivalent to the condition that $F$ is a $\Symp_n(\Q)$-translate of a standard boundary component.
We define the {\em rational closure} of $\calD_n$ by 
\begin{align}
\calD_n^*= \bigcup_{F} F \subset \bar\calD_n,
\end{align}
where the union extends over all rational boundary components (including the non-proper boundary component $\calD_n$). The action of $\Symp_n(\Q)$ on $\calD_n$ extends to an action on $\calD_n^*$.

\subsection{The cylindrical topology}
Let $0\leq m\leq n$. Recall that there is a holomorphic map
\[
\pi_{n,m}: \H_n\to \H_{m}\cong F_{m},\quad \pi_{n,m}\zxz{\tau_1}{\tau_{12}}{{}^t\tau_{12}}{\tau_2}= \tau_1,
\]
and a real analytic map
\[
\rho_{n,m}: \H_n\to \Sym_{n-m}^+(\R),\quad \rho_{n,m}\zxz{\tau_1}{\tau_{12}}{{}^t\tau_{12}}{\tau_2}= v_2-{}^t v_{12} v_1^{-1} v_{12},
\]
where $v$ denotes the imaginary part of $\tau$.
Both maps are equivariant for the action of the parabolic subgroup $G_{F_{m}}$.

%If $U\subset \H_{m}$ is open and $B\in\Sym_{n-m}^+(\R)$, we consider the open subset 
%\[
%N_n(U,B)=\{\tau \in \H_n\mid \; \text{$\pi_{n,m}(\tau)\in U$ and $\rho_{n,m}(\tau)-B>0$}\} 
%\]
%of $\H_n$. 

We now recall the definition of the {\em cylindrical topology} on $\calD_n^*$ following 
%Sometimes it is convenient to work with the following alternative base of the cylindrical  topology as in 
\cite[Chapter II.6]{Fr} (see also \cite[Section~5]{Na}).
For a matrix $v\in \Sym_{n}^+(\R)$ we define 
\[
\operatorname{m}(v) = \min_{\substack{x\in \Z^n\\ x\neq 0}}
 v[x],
\]
i.e., the minimum of the quadratic form $x\mapsto v[x] = {}^t x v x$ on non-zero integral vectors. If $U\subset \H_{m}$ is open and $C>0$, we consider the open subset 
\begin{align}
\label{eq:w}
W_n(U,C)=\{\tau \in \H_n\mid \; \text{$\pi_{n,m}(\tau)\in U$ and $\operatorname{m}(\rho_{n,m}(\tau))>C$}\} 
\end{align}
of $\H_n$. Note that when $n=m$, we simply have $W_n(U,C)=U$.
To define a basis of neighbourhoods of a point $z$ in the standard boundary component $F_{m}\cong \H_{m}$, we consider the chain of standard boundary components 
\[
\calD_n=F_n > F_{n-1} > \dots > F_{m+1}>F_{m}
\]
adjacent to $F_{m}$. Let $U\subset \H_{m}$ be an open neighbourhood of $z$. 
If $n\geq j\geq m$, we may view 
\[
W_j(U,C)\subset F_{j}
\]
as a subset via the identification $\H_{j}\cong F_{j} $. We use this to 
define a subset of $\calD_n^*$ by 
\begin{align}
\label{eq:wt}
\tilde W_n(U,C ) &= \Gamma_{n,m}^0 \left(\bigcup_{n\geq j\geq m} W_j(U,C)\right).
\end{align}
%where $W_j(U,C)$ is viewed as a subset of $F_{j}$.
Note that when $m=n$, then $\tilde W_n(U,C )$ simply reduces to $U$.

\begin{definition}
\label{def:cyltop}
A set $V\subset \calD_n^* $ is called open if 
%for all $z\in V$ the following condition is satisfied:
for all $z\in V$ there exists a $g\in \Symp_n(\Q)$ such that $gz$ is contained in a standard boundary component $F_{m}$ for some $0\leq m\leq n$ and such that 
$gV$ contains a set $\tilde W_n(U,C )$ for some open neighbourhood $U\subset F_{m}\cong \H_{m}$ of $gz$ and some $C>0$.
\end{definition}

\begin{proposition}
The cylindrical topology is the weakest topology on $\calD_n^*$ in which all $\Symp_n(\Q)$-translates of all the sets  $\tilde W_n(U,C)$ are open for $U\subset F_{m}$ open, $0\leq m\leq n$, and $C>0$.
The induced topology on the standard boundary components $F_{m}$ agrees with the usual topology. The set $\calD_n$ is open and dense in $\calD_n^*$. Moreover, $\calD_n^*$ is a Hausdorff space, and $\Symp_n(\Q)$ acts on it by homeomorphisms.
\end{proposition}

\begin{comment}
If we restricted in the definition of the $N_n(U,B_{0},\dots,B_{n-m} )$ to those where the $B_{j-m}$ are positive scalar matrices, we would obtain the same topology. This is easily seen using the action of elements 
\begin{align}
\label{eq:matg2}
\begin{pmatrix}
1 & 0 & 0 & 0\\
0 & u & 0 & 0\\
0 & 0 & 1 & 0\\
0 & 0 & 0 & {}^tu^{-1}
\end{pmatrix}
\end{align}
of $\Symp_n(\Z)\cap G_{F_{m}}^0$ for $u\in \GL_{n-m}(\Z)$, together with well known estimates for Minkowski reduced matrices in 
$\Sym_{n-m}^+(\R)$ as in \cite[Chapter I.2]{Fr}, see Folgerung 2.6.
\end{comment}

\begin{remark}
\label{rem:conv}
A sequence 
%\[
%\tau^{(\nu)}=\zxz{\tau_1^{(\nu)}}{\tau_{12}^{(\nu)}}{{}^t\tau_{12}^{(\nu)}}{\tau_2^{(\nu)}}
%\in \H_n\cong\calD_n
%\]
\[
\tau^{(\nu)}=\zxz{\tau_1^{(\nu)}}{*}{*}{*}
\in \H_n\cong\calD_n
\]
with $\tau_1^{(\nu)}\in \H_{m}$ converges to a boundary point $\tau_1^*\in \H_{m}\cong F_m$, 
if and only if $\tau_1^{(\nu)}\to \tau_1^*$ in the usual sense and 
$\rho_{n,m}(\tau^{(\nu)})\to \infty$. 
Here the latter condition means that for any $C>0$ we have 
$\operatorname{m}(\rho_{n,m}(\tau))>C$
for all but finitely many $\nu$. See e.g.~\cite[Hilfssatz $6.18_1$]{Fr}.
\end{remark}

\begin{comment}
\begin{remark}
The following alternative definition of the cylindrical topology is given in \cite[Section~5]{Na} and \cite[Chapter~6]{HKW}.
For $U\subset \H_{m}$ open and $C>0$, let  
\[
N_n(U,C)=\{\tau \in \H_n\mid \; \text{$\pi_{n,m}(\tau)\in U$ and $\rho_{n,m}(\tau)-C\cdot 1_{n-m}>0$}\}. 
\]
Define a subset of $\calD_n^*$ by 
\begin{align}
\label{eq:nt}
\tilde N_n(U,C ) &= \Gamma_{n,m}^0 \left(\bigcup_{n\geq j\geq m} N_j(U,C)\right),
\end{align}
where $N_j(U,C)$ is viewed as a subset of $F_{j}$.
Note that for $m=n$ the set $\tilde N_n(U,C )$ is simply $U$. 
The cylindrical topology on $\calD_n^*$ can also be described as the weakest topology on $\calD_n^*$ in which all $\Symp_n(\Q)$-translates of the sets  $\tilde N_n(U,C)$ are open for $U\subset F_{m}$ open, $0\leq m\leq n$, and $C>0$. The equivalence of the two definitions can be shown using Minkowski reduction theory as in \cite[Chapter I.2]{Fr}, see Folgerung 2.6.
\end{remark}
\end{comment}

The following lemmas will be used to construct convenient neighborhoods of boundary components.

\begin{lemma}
\label{lem:nbhd}
%Let $\Gamma$ be an arithmetic subgroup of $G$, and 
Let $U\subset \H_{m}$ be relatively compact. 

(i) There exists a $C>0$ such that every $\gamma\in \Gamma_n$ satisfying 
\[
\gamma (W_n(U,C))\cap W_n(U,C)\neq \emptyset
\]
is contained in $\Gamma_{n,m}$.

(ii)
Moreover, there exists a $C>0$ such that every $\gamma\in \Gamma_n$ satisfying 
\[
\gamma (\tilde W_n(U,C))\cap \tilde W_n(U,C)\neq \emptyset
\]
is contained in $\Gamma_{n,m}$.
\end{lemma}

\begin{proof}
(i) The following argument is due to Eberhard Freitag.
We decompose any  $z\in \H_n$ as
\begin{align}
\label{eq:dec}
z= \zxz{z_1}{z_{12}}{{}^tz_{12}}{z_2}
\end{align}
with $z_1\in \H_{m}$. 
We argue indirectly. Assume that there exists a sequence $C_\nu\to \infty $ of positive real numbers such that for every $\nu \in \Z_{>0}$ there exists a $\gamma_\nu\in \Gamma_n\setminus \Gamma_{n,m}$ and points $z^{(\nu)}, w^{(\nu)} \in W_n (U,C_\nu)$ satisfying 
\[
\gamma_\nu z^{(\nu)} = w^{(\nu)}.
\]  
By taking a suitable subsequence, we may assume that $z_1^{(\nu)}$ converges to a point $z_1\in \H_{m}$ for $\nu\to \infty$, and $w_1^{(\nu)}$ converges to a $w_1\in \H_{m}$. Hence, with respect to the cylindrical topology we get convergent sequences
\[
z^{(\nu)}\to z_1\in F_{m}, \qquad w^{(\nu)}\to w_1\in F_{m}.
\]

According to \cite[Hilfssatz $6.18_2$]{Fr} there exist elements $\alpha_\nu$ in the stabilizer of $z_1$ inside $\Gamma_n$ such that all 
\[
\tilde z^{(\nu)}=\alpha_\nu ( z^{(\nu)})
\]
are contained in some fixed Siegel domain $\calF_n(u)$. Analogously, there are $\beta_\nu$ in the stabilizer of $w_1$ inside $\Gamma_n$ such that all 
\[
\tilde w^{(\nu)}=\beta_\nu ( w^{(\nu)})
\]
lies in $\calF_n(u)$. Both stabilizers are contained in $\Gamma_{n,m}$. The construction in 
loc.\ cit.\
%\cite[Hilfssatz $6.18_2$]{Fr} 
shows that 
\[
z_1^{(\nu)} = \tilde z_1^{(\nu)},\qquad \operatorname{m}(\rho_{n,m}(z^{(\nu)})) = \operatorname{m}(\rho_{n,m}(\tilde z^{(\nu)})),
\]
and similarly for the $w^{(\nu)}$. Therefore, we may assume without loss of generality from the outset that the sequences $z^{(\nu)}$ and $w^{(\nu)}$ are contained in a fixed Siegel domain $\calF_n(u)$.

The finiteness property of Siegel domains now implies that the $\gamma_\nu$ belong to a finite set. By taking a subsequence we may assume that $\gamma=\gamma_\nu$ is independent of $\nu$.
Since $\gamma$ acts continuously on $\calD_n^*$ we find that 
\[
\gamma (z_1)=w_1.
\]
But this implies $\gamma\in \Gamma_{n,m}$, a contradiction.

%(ii) Assume that $z,w\in \tilde W_n(U,C)$ and $\gamma\in \Gamma$ such that $\gamma z = w$. We want to show that $\gamma\in \calP(F_{m})$ when $C$ is sufficiently large.
%By definition there exist $n\geq j,j'\geq m$ and $\gamma_1,\gamma_2\in \Lambda (F_{m})$ and $z'\in W_j(U,C)$, $w'\in W_{j'}(U,C)$, such that 
%\[
%z=\gamma_1 z',\qquad w=\gamma_2 w', \qquad \gamma\gamma_1 z' = \gamma_2 w'.
%\]
%Replacing $\gamma$ by $\gamma_2^{-1}\gamma \gamma_1$ we may assume that $\gamma_1=\gamma_2=1$ and $z=z'$, $w=w'$. Since the action of $\Gamma$ preserves the degree of a boundary component, we may further assume that $j=j'$. But then, according to \cite[Hilfssatz 2.5]{Fr}, the condition $\gamma z =w \in W_j(U,C)\subset F_j$ implies that $\gamma\in \calP(F_{j})$. 
%
%
%Hence it suffices to show that there is a $C>0$ such that every $\gamma\in \Gamma\cap \calP(F_{j})$ satisfying 
%\[
%\gamma (W_j(U,C))\cap W_j(U,C)\neq \emptyset
%\]
%is contained in $\calP(F_{m})$.
 %
%
%----

(ii) Part  (i) of the lemma immediately implies that there exists a $C>0$ such that for every $j$ with $n\leq j\leq m$ and every $\gamma\in \Gamma_{n,j}$ satisfying 
\[
\gamma (W_j(U,C))\cap W_j(U,C)\neq \emptyset
\]
we have $\gamma\in \Gamma_{n,m}$. We fix such a $C$ and assume that  $z,w\in \tilde W_n(U,C)$ and $\gamma\in \Gamma_n$ with the property that  $\gamma z = w$. 

By definition there exist $n\geq j,j'\geq m$ and $\gamma_1,\gamma_2\in \Gamma_{n,m}^0$  and $z'\in W_j(U,C)$, $w'\in W_{j'}(U,C)$, such that 
\[
z=\gamma_1 z',\qquad w=\gamma_2 w', \qquad \gamma\gamma_1 z' = \gamma_2 w'.
\]
Replacing $\gamma$ by $\gamma_2^{-1}\gamma \gamma_1$ we may assume that $\gamma_1=\gamma_2=1$ and $z=z'$, $w=w'$. Since the action of $\Gamma_n$ preserves the degree of a boundary component, we may further assume that $j=j'$. But then, according to \cite[Hilfssatz 2.5]{Fr}, the condition $\gamma z =w \in W_j(U,C)\subset F_j$ implies that $\gamma\in \Gamma_{n,j}$. Consequently, by our choice of $C$, we may conclude that $\gamma\in \Gamma_{n,m}$.
\end{proof}

We write $z=x+iy$ for the decomposition of $z\in \H_n$ into its real and imaginary part. Moreover, we denote the Jacobi decomposition of $y$ by
\begin{align}
\label{eq:jac}
y=D[W] = {}^t W D W,
\end{align}
where $D$ is a diagonal matrix with diagonal entries $d_1,\dots ,d_n$ and $W=(w_{ij})$ is a unipotent upper triangular matrix. Recall that for $u>0$ the {\em Siegel domain} $\calF_n(u)$ is defined as the set of $z\in \H_n$ satisfying the following conditions:
\begin{enumerate}
\item[(a)] 
$|x_{ij}|<u$ for all $1\leq i,j\leq n$,
\item[(b)] 
$|w_{ij}|<u$ for all $1\leq i<j\leq n$,
\item[(c)] 
$d_{i}<u d_{i+1}$ for all $1\leq i\leq n-1$,
\item[(d)] 
$1< u d_{1}$,
\end{enumerate}
see e.g.~\cite[Chapter II, Definition 1.7]{Fr}. 
The set of positive definite symmetric matrices $y\in \Sym_n^+(\R)$ satisfying conditions (b) and (c) is denoted by $\calR_n(u)$.
%
%For $u>0$ we denote by $\calF_n(u)\subset \H_n$ the corresponding Siegel domain as in Definition~1.7 in Chapter~II of \cite{Fr}. 
We also define 
\begin{align}
\label{eq:deffn*}
\calF_n^*(u) = \calF_n(u)\cup \calF_{n-1}(u)\cup \dots\cup \calF_0(u)\subset \calD_n^*
\end{align}
as in \cite[page 98]{Fr}.

\begin{lemma}
\label{lem:nbhd-siegel}
%Let $\Gamma$ be an arithmetic subgroup of $G$, and 

Let $U\subset \H_{m}$ be relatively compact. 

(i) There exists a $C>0$ such that every $\gamma\in \Gamma_n$ satisfying 
\[
\gamma (W_n(U,C))\cap \calF_n(u)\neq \emptyset
\]
is contained in $\Gamma_{n,m}$.

(ii)
Moreover, there exists a $C>0$ such that every $\gamma\in \Gamma_n$ satisfying 
\[
\gamma (\tilde W_n(U,C))\cap  \calF_n^*(u)\neq \emptyset
\]
is contained in $\Gamma_{n,m}$. 
\end{lemma}

\begin{proof}
This can be proved in the same way as Lemma \ref{lem:nbhd}. Therefore we omit the details.
\end{proof}

We now fix a $u>0$ such that $\calF_j(u)$ is a fundamental set for the action of $\Symp_j(\Z)$ on $\H_j$ for every $0\leq j\leq n$.

\begin{lemma}
\label{lem:nbhd2}
Let $U\subset \calF_m(u)\subset F_m$ be relatively compact. 

(i) There exists a $C>0$ such that 
\[
W_n(U,C)\subset \Gamma_{n,m}^0\calF_n(u).
\]

(ii) There exists a $C>0$ such that 
\[
\tilde W_n(U,C)\subset \Gamma_{n,m}^0\calF_n^*(u).
\]
\end{lemma}

\begin{proof}
We write $z=x+iy\in \H_n$ for the decomposition in real and imaginary part. 
%Moreover, we denote the Jacobi decomposition of $y$ by
%\begin{align}
%y=D[W] = {}^t W D W,
%\end{align}
%where $D$ is a diagonal matrix with diagonal entries $d_1,\dots ,d_n$ and $W=(w_{ij})$ is a unipotent upper triangular matrix. Recall that the Siegel domain $\calF_n(u)$ is defined as the set of $z\in \H_n$ satisfying the following conditions:
%\begin{enumerate}
%\item[(a)] 
%$|x_{ij}|<u$ for all $1\leq i,j\leq n$,
%\item[(b)] 
%$|w_{ij}|<u$ for all $1\leq i<j\leq n$,
%\item[(c)] 
%$d_{i}<u d_{i+1}$ for all $1\leq i\leq n-1$,
%\item[(d)] 
%$1< u d_{1}$,
%\end{enumerate}
%see e.g.~\cite[Chapter II, Definition 1.7]{Fr}. 
%
To prove the lemma, we use the block Jacobi decomposition 
\begin{align*}
y= D[W]=  \zxz{Y_1}{0}{0}{Y_2}\left[ \zxz{1_m}{B}{0}{1_{n-m}}\right],
\end{align*}
and put 
\[
D=\zxz{D_1}{0}{0}{D_2}, \qquad W=\zxz{W_1}{W_{12}}{0}{W_2},
\]
where $D_1$ is the diagonal matrix with diagonal entries $d_1,\dots,d_m$, and $D_2$  the diagonal matrix with entries $d_{m+1},\dots,d_n$. The matrices $W_1$ and $W_2$ are unipotent upper triangular. It is easily seen that $Y_1= D_1[W_1]$, $Y_2=D_2[W_2]$, and $B=W_1^{-1} W_{12}$.

(i) Let $z\in W_n(U,C)$. We have to show that there exists a $\gamma\in \Gamma_{n,m}^0$ such that $\gamma z\in \calF_n(u)$. We note that since $z_1\in U$, we have 
$y_1=Y_1= D_1[W_1]\in \calR_m(u)$.
We may act with a matrix in $\Gamma_{n,m}^0$ of the form 
\[
\zxz{{}^tA}{0}{0}{A^{-1}}, \qquad A=\zxz{1_m}{0}{0}{A_2},
\]
where $A_2\in \GL_{n-m}(\Z)$, to transform $Y_2$ to $ Y_2[A_2]\in \calR_{n-m}(u)$. 
Here $Y_1$ remains unchanged. Next we can act 
with a matrix in $\Gamma_{n,m}^0$ of the form 
\[
\zxz{{}^tA}{0}{0}{A^{-1}}, \qquad A=\zxz{1_m}{A_{12}}{0}{1_{n-m}},
\]
where $A_{12}\in \Z^{m\times(n-m)}$, to transform $W_{12}$ to $ W_{12}'=W_1 A_{12} + W_{12}$ with entries $w_{ij}'$ satisfying $|w_{ij}'|<u$. Here $D$, $W_1$, $W_2$ remain unchanged.

Finally, we can act with a matrix 
in $\Gamma_{n,m}^0$ of the form 
\[
M=\zxz{1_n}{T}{0}{1_n}, \qquad T=\zxz{0_m}{T_{12}}{{}^tT_{12}}{T_2},
\]
where $T_{12}\in \Z^{m\times(n-m)}$, $T_{2}\in \Z^{(n-m)\times(n-m)}$, to transform $x_{12}$ to $x_{12} + T_{12}$ and  $x_{2}$ to $x_{2} + T_{2}$. Here $x_1$ and $y$ remain unchanged.

In this way we can transform $z\in W_n(U,C)$ by an element of $\Gamma_{n,m}^0$ to a $z'\in \H_n$ meeting all conditions for $\calF_n(u)$ except for possibly the condition 
\begin{align}
\label{eq:remc}
d_m'<ud_{m+1}'.
\end{align}
Here and throughout we have indicated the Jacobi coordinates of $z'$ by $d_i'$, $w_{ij}'$, and $x_{ij}'$.
We claim that the condition \eqref{eq:remc} 
is also met if we choose $C$ sufficiently large. 
To see this we note that $d_m'$ is bounded, since $z_1'=z_1$ is contained in the bounded set $U$.
Moreover, since $Y_2'= Y_2[A_2]\in \calR_{n-m}(u)$, we have 
%there exists an $\eps>0$ (only depending on $n$ and $m$) such that 
\[
%\eps d_{m+1}\leq 
\operatorname{m}(Y_2')\leq d_{m+1}'.
\]
see e.g.~\cite[Hilfssatz 1.2]{Fr}. Hence, if 
\[
C>\frac{1}{u}\sup\{ d_m\mid z_1\in U\},
\]
and $z\in W_n(U,C)$ then 
\[
\frac{1}{u}d_{m}' < C <\operatorname{m}(Y_2)=\operatorname{m}(Y_2')\leq d_{m+1}'.
\]
This gives the remaining condition \eqref{eq:remc}.

(ii) This assertion is an immediate consequence of (i).
\end{proof}

% Better version
%\begin{lemma}
%\label{lem:nbhd}
%Let $\Gamma$ be an arithmetic subgroup of $G$, and let $U\subset \H_{m}$ be relatively compact. There exists a $C>0$ such that every $\gamma\in \Gamma$ satisfying 
%\[
%\gamma (W_n(U,C))\cap W_n(\H_{m},C)\neq \emptyset
%\]
%is contained in $\calP(F_{m})$.
%Moreover, there exists a $C>0$ such that every $\gamma\in \Gamma$ satisfying 
%\[
%\gamma (\tilde W_n(U,C))\cap \tilde W_n(\H_{m},C)\neq \emptyset
%\]
%is contained in $\calP(F_{m})$. 
%\end{lemma}
%

\subsection{The Satake compactification}

Let $\Gamma\subset \Symp_n(\Q)$ be an {\em arithmetic} subgroup, that is, a subgroup which is commensurable with $\Gamma_n=\Symp_n(\Z)$. Then $\Gamma$ acts properly discontinuously on $\calD_n^*$. The quotient $X_\Gamma^*:=\Gamma\bs \calD_n^*$, equipped with the quotient topology, is a compact Hausdorff space, which contains $X_\Gamma:=\Gamma\bs \calD_n$ as a dense open subset. 
The complex structure on $X_\Gamma$ canonically extends to a complex structure on $X_\Gamma^*$, equipping it with the structure of a normal complex space. It is called the Satake compactification of $X_\Gamma$, see e.g.~\cite[Theorem 10.4]{BB},
%\cite[Chapter~I.6]{HKW}, 
\cite[Section~5]{Na}, or \cite[Chapter~II.6]{Fr}.

The boundary 
\begin{align}
\label{eq:boundary}
\partial X_\Gamma^*= X_\Gamma^*\setminus X_\Gamma
\end{align}
is a closed analytic subset  of codimension $n$. If $F$ is a rational boundary component of $\calD_n$ of degree $m$, we let 
\begin{align*}
\Gamma_F&= \Gamma\cap G_F,\\
\Gamma_F^0&= \Gamma\cap G_F^0,\\
\Gamma_{F}' &= \Gamma\cap G_{F}',\\
\Gamma_{F}^J &= \Gamma\cap G_{F}^J.
\end{align*}
We may view 
\[
\bar\Gamma_F= \Gamma_F / \Gamma_F^0
\]
as an arithmetic subgroup of $\Symp_{m}(\R)$.
The quotient 
\[
X_{\Gamma,F} = \bar \Gamma_F\bs F
\]
is isomorphic to a Siegel modular variety of genus $m$.
%(of smaller genus if $F$ is a proper boundary component).
The boundary decomposes as a finite disjoint union 
\[
\partial X_\Gamma^*= \coprod_{F} X_{\Gamma,F}
\]
of locally closed analytic subsets, where the union runs over the $\Gamma$-classes of proper rational boundary components, see \cite[Corollary 4.11]{BB}. The following two propositions ensure the existence of convenient neighborhoods of the boundary components.

\begin{proposition}
\label{prop:boundary0}
%Let $\Gamma\subset \Gamma_n$ be a congruence subgroup. 
For every rational boundary component $F$ there exists an open neighborhood  $V(F)\subset \calD_n^*$ of $F$ satisfying the following properties:
\begin{itemize}
\item[(i)]
$V(F)$ is invariant under the stabilizer $\Gamma_{n,F}$ of $F$ in $\Gamma_n$.
\item[(ii)]
The natural map 
\[
\Gamma_{n,F}\bs V(F)\to X_{\Gamma_n}^*
\]
is injective.
\item[(iii)]
$V(\delta F)= \delta V(F)$ for all $\delta \in \Gamma_n$.
\item[(iv)]
If $F$ and $F'$ are rational boundary components of degree $m$, we have 
\[
V(F)\cap V(F')\neq\emptyset \quad \Rightarrow\quad F=F'.
\]
\end{itemize}
\end{proposition}

\begin{proof}
We first prove (i) and (ii) for the standard boundary components.
Let $m$ be an integer with $0\leq m< n$ and consider the standard boundary component $F_{m}$ of degree $m$. Recall the notation $\Gamma_{n,F_m}= \Gamma_{n,m}$.

Let $u>0$ be such that $\calF_j(u)$ is a fundamental set for the group $\Gamma_j$ for all $0\leq j\leq n$. %Then the natural map $\tilde\calF_n(u)\to  X_{\Gamma}^*$ is surjective.
Choose a sequence $U_\nu\subset \calF_m(u)$ of relatively compact open sets (for $\nu \in \Z_{>0}$) such that 
\[
\calF_m(u)= \bigcup_{\nu\geq 1} U_\nu.
\]
According to  Lemma \ref{lem:nbhd2} and Lemma \ref{lem:nbhd-siegel}, we may choose  $C_\nu>0$ such that
\begin{align}
\label{eq:keycond0}
\tilde W_n(U_\nu,C_\nu)\subset \Gamma_{n,m}^0\calF_n^*(u),
\end{align}
and such that every $\gamma\in \Gamma_n$ satisfying 
\begin{align}
\label{eq:keycond}
\gamma (\tilde W_n(U_\nu,C_\nu))\cap  \calF_n^*(u)\neq \emptyset
\end{align}
is contained in $\Gamma_{n,m}$.

The union 
\begin{align*}
S(F_m) &= \bigcup_{\nu\geq 1} \tilde W_n(U_\nu,C_\nu)\subset \calD_n^*
\end{align*}
is an open neighborhood of $\calF_m(u)$. By the choice of $u$, the set
\begin{align}
\label{eq:defWFm} 
V(F_m) &:= \Gamma_{n,m} (S(F_m))
\end{align}
is an open neighborhood of the full boundary component $F_m$. Moreover, by construction, $V(F_m)$ is invariant under $\Gamma_{n,F_m}$. 

We now show the injectivity of the natural map $\Gamma_{n,m}\bs V(F_m)\to X^*_{\Gamma_n}$. To this end, 
let $z,w\in V(F_m)$ and $\gamma\in \Gamma_n$ such that 
\[
\gamma z = w.
\]
We have to show that 
%Since $\Gamma_{F_m}= \Gamma\cap \Gamma_{n,m}$,  it suffices to show that 
$\gamma\in \Gamma_{n,m}$.

Possibly shifting $z$ and $w$ by elements of $\Gamma_{n,m}$, we can assume that $z$ and $w$ lie in $S(F_m)$. By \eqref{eq:keycond0} we can further assume that they are also contained in $\calF_n^*(u)$. Hence there exist  $\mu,\nu\in \Z_{>0}$ such that 
\begin{align*}
z&\in \tilde W_n(U_\nu,C_\nu)\cap  \calF_n^*(u),\\
w&\in\tilde W_n(U_\mu,C_\mu)\cap  \calF_n^*(u).
\end{align*}
Using condition \eqref{eq:keycond}
we see that $\gamma\in \Gamma_{n,m}$.
This concludes the proof of (i) and (ii) for the standard boundary component $F_m$.

If $F$ is any rational boundary component of degree $m$, we choose $\delta\in \Gamma_n$ such that $F=\delta F_m$ and put 
\[
V(F) = \delta V(F_m).
\]  
By property (i) for $V(F_m)$ this is independent of the choice of $\delta$.
Employing the fact that 
%$(\delta^{-1}\Gamma \delta)_{F_m} = \delta^{-1}\Gamma_F \delta$, 
$\Gamma_{n,F_m} = \delta^{-1}\Gamma_{n,F} \delta$, 
it is easily seen that $V(F)$ satisfies properties (i) and (ii) for $F$.
Moreover, in this way (iii) also holds.

Finally, to prove (iv), let $F$ and $F'$ be rational boundary components of degree $m$, and let $z\in V(F)\cap V(F')$. Choose $\delta,\delta'\in \Gamma_n$ such that $F=\delta F_m$ and $F'=\delta' F_m$. Since $V(F)=\delta V(F_m)$ and $V(F')= \delta' V(F_m)$,  there are $z_1, z_2\in V(F_m)$ such that 
\[
z=\delta z_1 = \delta' z_2.
\]
Then $\delta^{-1} \delta' z_2= z_1$, and by (ii) we obtain $\delta^{-1} \delta'\in \Gamma_{n,m}$.
Consequently,
\[
F'= \delta' F_m = \delta (\delta^{-1} \delta') F_m = \delta F_m = F.
\]
This concludes the proof of the proposition.
\end{proof}

\begin{comment}
\begin{remark}
\label{rem:bound1}
Let $\Gamma\subset \Gamma_n= \Symp_n(\Z)$ be a congruence subgroup. The open neighborhoods $W(F)$ of Proposition \ref{prop:boundary} can be chosen such that we have in addition:
\[
W(F)\cap W(F')\neq\emptyset \quad \Rightarrow\quad F=F',
\]
for any two proper rational boundary components of degree $m$. 
\end{remark}

\begin{proof}
We define $W(F_m)$ as in \eqref{eq:defWFm}. For any rational boundary component $F$ of degree $m$ we choose $\delta\in \Gamma_n$ such that $F=\delta F_m$ and put 
\[
W(F) = \delta W(F_m).
\]  
Using Proposition \ref{prop:boundary}, it is easily checked that all claimed properties hold. 
\end{proof}
\end{comment}

\begin{proposition}
\label{prop:boundary}
Let $\Gamma\subset \Symp_n(\Q)$ be an arithmetic subgroup. 
%Then the assertion of Proposition~\ref{prop:boundary} also holds for $\Gamma$. 
For every rational boundary component $F$ there exists an open neighborhood  $W(F)\subset \calD_n^*$ of $F$ satisfying the following properties:
\begin{itemize}
\item[(i)]
$W(F)$ is invariant under the action of  $\Gamma_F$.
\item[(ii)]
The natural map 
\[
\Gamma_F\bs W(F)\to X_\Gamma^*
\]
is injective.
\item[(iii)]
$W(\gamma F)= \gamma W(F)$ for all $\gamma \in \Gamma$.
\item[(iv)]
If $F$ and $F'$ are rational boundary components of degree $m$, we have 
\[
W(F)\cap W(F')\neq\emptyset \quad \Rightarrow\quad F=F'.
\]
\end{itemize}
\end{proposition}

\begin{proof}
In the special case when $\Gamma\subset \Gamma_n$, it is easily seen that we may simply put $W(F)=V(F)$ with $V(F)$ as in Proposition \ref{prop:boundary0}.

Now let $\Gamma\subset \Symp_n(\Q)$ be an arbitrary arithmetic subgroup. We choose a congruence subgroup 
\[
\Gamma'\subset \Gamma\cap \Gamma_n
\]
which is normal in $\Gamma$. For any rational boundary component $F$, we put
\[
W(F) = \bigcap_{\delta\in \Gamma} \delta^{-1} V(\delta F).
\]
Since $\delta^{-1} V(\delta F)=V(F)$ for $\delta \in \Gamma'$,  
%because of Proposition \ref{prop:boundary0} (iii) 
this is in fact an intersection of  finitely many different open neighborhoods of $F$ in $\calD_n^*$. Hence it defines an open neighborhood of $F$. We leave it the the reader to verify that properties (i)-(iv) hold.
\begin{comment}
For $\gamma \in \Gamma$ we have 
\begin{align*}
\gamma W(F) &= \bigcap_{\delta\in \Gamma} \gamma\delta^{-1} V(\delta F)\\
&= \bigcap_{\delta\in \Gamma} \delta^{-1} V(\delta\gamma F)\\
&= W(\gamma F),
\end{align*}
which shows (iii).
Moreover, if $\gamma\in \Gamma_F$ then $\gamma W(F)=W(F)$. So (i) holds.

To see (ii), let $z,w\in W(F)$ and let $\gamma\in \Gamma$ with $\gamma z=w$. For any $\delta\in \Gamma$ we can write $z=\delta^{-1}z_\delta$ for some point $z_\delta \in V(\delta F)$. Then 
\[
\gamma \delta^{-1}z_\delta = w\in W(F)\subset V(F).
\]
If we take $\delta = \gamma$, we find $z_\gamma= w\in V(\gamma F)\cap V(F)$. By means of Proposition \ref{prop:boundary0} (iv), we obtain $\gamma F=F$ and therefore $\gamma \in \Gamma_F$.
 
For (iv) we notice that $W(F)\subset V(F)$ for all rational boundary components $F$. Hence $W(F)\cap W(F')\subset V(F)\cap V(F')$, and the assertion directly follows from part  (iv) of Proposition \ref{prop:boundary0}.  
\end{comment}
\end{proof}

\begin{remark}
We may in addition require that the open neighborhood $W(F)$ in Proposition~\ref{prop:boundary} is connected. This follows by replacing $W(F)$ by its connected component containing $F$ if necessary.
\end{remark}

\subsection{Siegel modular forms}
Let $k$ be an integer. We denote the usual action 
of $G=\Symp_n(\R)$ in weight $k$ on functions $f:\H_n\to \C$ by 
\[
(f\mid_k \gamma)(\tau) = \det (c\tau +d)^{-k} f(\gamma \tau)
\]
for $\gamma=\kabcd \in G$.
The composition of the quotient map and the inclusion defines a natural holomorphic map
\[
p:\H_n\to X_\Gamma^*.
\]
We denote by $\omega$ the sheaf of modular forms of weight $1$ on $X_\Gamma^*$.
Recall that if $V\subset X_\Gamma^*$ is open, the module of sections $\omega(V)$ is given by holomorphic functions $f:p^{-1}(V)\to \C$ satisfying 
\[
f\mid_1 \gamma= f 
\]
for all $\gamma\in \Gamma$ and all $\tau \in p^{-1}(V)$. Moreover, when $n=1$, it is also required that $f$ is holomorphic at the cusps contained in $V$.
If $n>1$, then regularity at the boundary is automatically satisfied by the local Koecher principle.

The sheaf $\omega$ is a coherent $\calO_{X_\Gamma^*}$-module on $X_\Gamma^*$, which can be identified with the Hodge bundle. The global sections of $\omega^{\otimes k}$ are given by holomorphic modular forms of weight $k$ for $\Gamma$, see e.g.~\cite[Section 10]{BB}.

We now describe the sheaf $\omega^{\otimes k}$ near a rational boundary component $F$.
By possibly conjugating with an element of $\Gamma_n$ it suffices to do this near the standard boundary components $F_m$, where $0\leq m\leq n$. 
Let $\tau_1\in F_m$ be a boundary point.
According to Proposition \ref{prop:boundary} and Definition \ref{def:cyltop} there exists a relatively compact open neighborhood $U\subset F_m$ of $\tau_1$ and a $C>0$ such that the open neighborhood $\tilde W_n(U,C)\subset \calD_n^*$ defined in \eqref{eq:wt} satisfies 
\begin{align}
\label{eq:h1}
\gamma \tilde W_n(U,C)\cap \tilde W_n(U,C)\neq \emptyset \quad \Rightarrow\quad \gamma\in \Gamma_{F_m}
\end{align}
for $\gamma\in \Gamma$. If $U$ is chosen sufficiently small, then the condition in \eqref{eq:h1} actually implies that $\gamma$ is contained in the stabilizer $\Gamma_{\tau_1}\subset \Gamma$ of $\tau_1$. This follows from the fact that $\Gamma_{F_m}/\Gamma_{F_m}^0$ acts properly discontinuously on the boundary component $F_m$. Possibly replacing $U$ by a smaller open set we may further attain that $\tilde W_n(U,C)$ is invariant under $\Gamma_{\tau_1}$.
Then 
\begin{align}
\label{eq:U}
V= \Gamma_{\tau_1}\bs \tilde W_n(U,C)\subset X_\Gamma^*
\end{align}
is an open neighborhood of $\tau_1$.
%Note that $\Gamma_{\tau_1}$ is contained in $\Gamma_{F_m}$.
%

An element $f\in \omega^{\otimes k}(V)$ is given by a continuous function $f:\tilde W_n(U,C)\to \C$ which is holomorphic on $W_n(U,C)$ and satisfies $f\mid_k \gamma = f$ for all $\gamma\in \Gamma_{\tau_1}$. In particular,  $f$ is invariant under the action of translations of the form
\begin{align}
\label{eq:center} 
\begin{pmatrix}
1 & 0 & 0 & 0\\
0 & 1 & 0 & s\\
0 & 0 & 1 & 0\\
0 & 0 & 0 & 1
\end{pmatrix}\in \Gamma_{F_m}',
\end{align}
where $s\in \Sym_{n-m}(\Z)$. Therefore $f$ has a partial Fourier expansion 
\begin{align}
\label{eq:fj}
f(\tau) = \sum_{T_2\in \Sym_{n-m}(\Q)} \phi_{T_2}(\tau_1,\tau_{12}) \,e(\tr T_2 \tau_2),
\end{align}
which converges normally in a small neighborhood of $U$.
Here the coefficients $\phi_{T_2}$ vanish unless $T_2$ is contained in a sublattice of  $\Sym_{n-m}(\Q)$ with bounded denominators (the character lattice of the torus $G_{F_m}'/\Gamma_{F_m}'$). Moreover, $\phi_{T_2}$ vanishes  if $T_2$ is not positive semi-definite.  We will refer to expansions as in \eqref{eq:fj} as Fourier-Jacobi expansions. 
%The equivariance of $f$ under the action of the  stabilizer of $\tau_1$ implies that the coefficients have the transformation behavior
%\begin{align}
%\label{eq:fjt}
%\phi_{T_2}(\tau_1,\tau_{12}) e(\tr T_2 \tau_2) \mid_k \gamma = \phi_{T_2}(\tau_1,\tau_{12}) e(\tr T_2 \tau_2)
%\end{align}
%for all $\gamma\in \Gamma_{\tau_1}$. In particular,  
The transformation behavior of $f$ under matrices of the form 
\begin{align}
\label{eq:matg22}
\begin{pmatrix}
1 & 0 & 0 & 0\\
0 & u & 0 & 0\\
0 & 0 & 1 & 0\\
0 & 0 & 0 & {}^tu^{-1}
\end{pmatrix}\in \Gamma_{F_m}^0
\end{align}
% of $\Symp_n(\Z)\cap G_{F_{m}}^0$ for $u\in \GL_{n-m}(\Z)$, 
implies that 
\begin{align}
\label{eq:sym0}
 \phi_{T_2[{}^t u^{-1}]} (\tau_1, \tau_{12} u) =  \det(u)^k\cdot
\phi_{T_2} (\tau_1, \tau_{12} ) 
%
%a\zxz{T_1}{T_{12}u}{{}^tu{}^tT_{12}}{T_2[u]} =\det(u)^k\cdot  a\zxz{T_1}{T_{12}}{{}^tT_{12}}{T_2} 
\end{align}
for all $u\in \GL_{n-m}(\Q)$ such that \eqref{eq:matg22} belongs to $\Gamma_{F_m}^0$.
%Conversely, any holomorphic function 

Now choose a small open neighborhood $W(F_m)$ of the {\em full} boundary component $F_m$ as in Proposition \ref{prop:boundary}. Then we easily obtain the following lemma.

\begin{lemma}
\label{lem:sect1}
%Let $U$ be as in \eqref{eq:U2}. 
Let 
\begin{align*}
%\label{eq:U2}
V=\Gamma_{F_m}\bs W(F_m)\subset X_\Gamma^*
\end{align*}
be the open neighborhood of the boundary stratum $X_{\Gamma,F_m}$ induced by $W(F_m)$. 
The space of sections $\omega^{\otimes k}(V)$ is given by those continuous functions 
$f:W(F_m)\to \C$ which are holomorphic on $W(F_m)\cap\calD_n$ and satisfy $f\mid_k \gamma =f $ for all  $\gamma\in \Gamma_{F_m}$. The Fourier-Jacobi coefficients $\phi_{T_2}$ of $f$ in the expansion \eqref{eq:fj} have the transformation behavior 
\begin{align}
\label{eq:fjt}
\phi_{T_2}(\tau_1,\tau_{12}) e(\tr T_2 \tau_2) \mid_k \gamma = \phi_{T_2}(\tau_1,\tau_{12}) e(\tr T_2 \tau_2)
\end{align} 
for all $\gamma$ in the Jacobi group $\Gamma_{F_m}^J = \Gamma\cap G_{F_m}^J$.
%\in \Gamma_{F_m}$ as in \eqref{eq:GFm} with $u=1$.
\end{lemma}

%The equivariance of $f$ under the action of the full stabilizer of $F_m$ 
%implies that the coefficients $\phi_{T_2}$ have the transformation behavior 
%\eqref{eq:fjt} 
%\begin{align*}
%\phi_{T_2}(\tau_1,\tau_{12}) e(\tr T_2 \tau_2) \mid_k \gamma = \phi_{T_2}(\tau_1,\tau_{12}) e(\tr T_2 \tau_2)
%\end{align*}
%for all $\gamma\in \Gamma_{F_m}$. 
Hence the coefficient $\phi_{T_2}$ is a weakly holomorphic Jacobi form of weight $k$ and index $T_2$ for an arithmetic subgroup of $\Symp_m(\Q)\ltimes \Mat_{m,n-m}(\Q)^2$ which only depends on $\Gamma$ but not on $T_2$.
\begin{comment}
Since any $f\in \omega^{\otimes k}(V)$ is also invariant under the action of translations of the form
\begin{align}
\label{eq:transl} 
\begin{pmatrix}
1 & 0 & b & r\\
0 & 1 & {}^tr & 0\\
0 & 0 & 1 & 0\\
0 & 0 & 0 & 1
\end{pmatrix}\in \Gamma_{F_m},
\end{align}
it has a Fourier expansion of the form
\begin{align}
\label{eq:fou}
f(\tau) = \sum_{T\in \Sym_{n}(\Q)} a(T)  \,e(\tr T \tau).
\end{align}
In particular the coefficients $\phi_{T_2}$  have the Fourier expansion 
\[
\phi_{T_2}(\tau_1,\tau_{12}) = \sum_{T_1 \in \Sym_{m}(\Q)} \sum_{T_2\in\Mat_{m,n-m}(\Q)} a\zxz{T_1}{T_{12}}{{}^tT_{12}}{T_2} e(\tr T_1 \tau_1 + 2\tr {}^tT_{12} \tau_{12}).
\]
\end{comment}

\section{Formal Siegel modular forms}
\label{sect:3}

We define formal Siegel modular forms for arithmetic subgroups of $\Symp_n(\Q)$. We interpret these as global sections of the formal completion of the sheaf Siegel modular forms along the Baily-Borel boundary.

\subsection{Formal Fourier-Jacobi series}

\label{sect:3.1}

Let $\Gamma\subset\Symp_n(\Q)$ be an arithmetic subgroup, and fix $k\in \Z$.
Let $F$ be a rational boundary component of degree $m$, where
$0\leq m\leq n$.
The center $G_F'$ of the unipotent radical of $G_F$ is isomorphic to the additive group of $\Sym_{n-m}(\R)$, and $\Gamma_F'$ defines a lattice $L_F$ in $G_F'$. Its dual lattice $L_F^\vee$ is the character lattice of the compact abelian group $G_F'/\Gamma_F'$.

\begin{definition}
\label{def:ffjF}
A {\em formal Fourier-Jacobi series} of weight $k$ for the boundary component $F$ and the group $\Gamma$ is a formal series 
\begin{align}
\label{eq:fsgen}
f(\tau)= \sum_{\substack{t\in L_F^\vee\\ t\geq 0}} f_t(\tau),
\end{align}
where the coefficients $f_t$ are holomorphic functions on $\H_n$ satisfying the transformation laws 
\begin{align}
\label{eq:tr1}
f_t\mid_k g & = t(g) \cdot f_t \qquad \text{for all $g\in G_F'$,}\\ 
\label{eq:tr2}
f_t\mid_k \gamma & = f_{p_\ell(\gamma)(t)}  \qquad\text{for all $\gamma\in \Gamma_F$.}
\end{align}
Here $p_\ell(\gamma)(t)$ denotes the action of $\gamma$ on $t\in L_{F}^\vee$ induced by \eqref{eq:pl}. In particular, for $\gamma\in \Gamma_F^J$ the condition in \eqref{eq:tr2} reduces to $f_t\mid_k \gamma  = f_{t}$.
\end{definition}

\begin{remark}
\label{rem:class0}
Let $F'$ be a second rational boundary component of degree $m$ and let $\delta\in \Gamma_n$ 
such that $F'=\delta^{-1}F$.  Then conjugation $g\mapsto \varphi_\delta (g)= \delta g\delta^{-1}$ with $\delta$ defines an automorphism of $G$, which restricts to an isomorphism $G_{F'}\to G_F$. In particular we have $G_{F'}=\delta^{-1}G_F\delta$ and   $(\delta^{-1}\Gamma\delta)_{F'}= \delta^{-1}\Gamma_F\delta$. 
The map $\varphi_\delta$ induces an isomorphism $G_{F'}'/(\delta^{-1}\Gamma\delta)_{F'}'\to G_F'/\Gamma_F'$. Pull back via this isomorphism gives rise to an isomorphism 
$\varphi_\delta^*: L_F^\vee\to L_{F'}^\vee$, $t\mapsto \varphi_\delta^*(t) = t\circ \varphi_\delta$ of the corresponding character lattices.
%$L_{F'}^\vee\to L_F^\vee$.
We define the pull back  $f\mid_k \delta$ of the formal Fourier-Jacobi series $f$ in \eqref{eq:fsgen} by defining the coefficients as
\begin{align}
\label{eq:pb0}
(f\mid_k \delta)_{\varphi^*_\delta(t)}:= f_t\mid_k \delta
\end{align}
for $t\in  L_F^\vee$. It is easily checked that 
\begin{align}
\label{eq:pb}
(f\mid_k \delta) (\tau)= \sum_{t\in L_F^\vee} (f_t\mid_k \delta)(\tau)
\end{align}
defines a formal Fourier-Jacobi series of weight $k$ for the boundary component $F'$ and the conjugate group $\delta^{-1} \Gamma\delta $. That is, we have 
\begin{align*}
f_t\mid_k \delta \mid_k g & = \varphi_\delta^*(t)(g) \cdot f_t \mid_k \delta
\qquad \text{for all $g\in G_{F'}'$,}\\ 
f_t\mid_k \delta \mid_k \gamma &= f_{p_\ell(\gamma)(\varphi^*_\delta (t))}  \mid_k \delta
 \qquad\text{for all $\gamma\in (\delta^{-1} \Gamma\delta)_{F'}$.}
\end{align*}
This definition is compatible with the pull-back of convergent Fourier-Jacobi series.
\end{remark}
%In particular, we obtain the following remark.

\begin{remark}
\label{rem:class}
Assume that $\delta\in \Gamma_n$. Then $f$ is a formal Fourier-Jacobi series of weight $k$ for the boundary component $F$ and the group $\Gamma$ if and only if  $f\mid_k \delta$ is a formal Fourier-Jacobi series of weight $k$ for the boundary component $\delta^{-1} F$ and the group $\delta^{-1}\Gamma\delta$.
\end{remark}

Let $F$ be a rational boundary component of degree $m$ and assume that $E<F$ is a rational boundary component of degree $0$ which is adjacent to $F$. Then we have inclusions of groups
\[
\xymatrix{
G_{F}'\ar[r] & G_{E}' \ar[r]   & G_F\cap G_E\\
\Gamma_{F}'\ar[r] \ar[u]& \Gamma_{E}' \ar[r]  \ar[u] & \Gamma_F\cap \Gamma_E\ar[u]},
\]
which induce an inclusion $G_F'/\Gamma_F'\to G_E'/\Gamma_E'$ and a surjective homomorphism  of the  character lattices $L_E^\vee\to L_F^\vee$, $s\mapsto s\mid L_F$.

Let $f=\sum_{t\in L_F^\vee}f_t$ be a formal Fourier-Jacobi series of weight $k$ for the boundary component $F$ and the group $\Gamma$. The transformation law \eqref{eq:tr2} implies that $f_t$ has a normally  convergent Fourier expansion 
\begin{align}
\label{eq:fj-fe1}
f_t = \sum_{\substack{s\in L_E^\vee \\ s\mid L_F =t}} f_{t,s},
\end{align}
where the coefficients $f_s:= f_{s\mid L_F,s}$ are holomorphic functions on $\H_n$ satisfying the transformation law
\begin{align}
\label{eq:tr11}
f_{s}\mid_k g & = t(g) \cdot f_{s} 
\end{align}
for all $g\in G_E'$ and $s\in L_E^\vee$. Putting these expansions together, we obtain the {\em formal Fourier expansion} 
\begin{align}
\label{eq:fouf1}
f = \sum_{\substack{s\in L_E^\vee }} f_{s}
\end{align}
of $f$ at $E$.

We call the formal Fourier-Jacobi series $f$ {\em regular} at the boundary component $E$ if its formal Fourier expansion \eqref{eq:fouf1}  defines a formal Fourier-Jacobi series of weight $k$ for the boundary component $E$ in the sense of Definition \ref{def:ffjF}. That is, 
$f_{s}=0$ unless $s\geq 0$, and 
\begin{align}
\label{eq:sym1}
f_{s}\mid_k \gamma & = f_{p_\ell(\gamma)(s)}  \qquad\text{for all $\gamma\in \Gamma_E$.}
\end{align}

\begin{remark}
\label{rem:class2}
Assume the above notation. 
If $E'<F$ is another rational boundary component of degree $0$ which is adjacent to $F$,  %Let $f$ be a formal Fourier-Jacobi series of weight $k$ for the boundary component $F$ and the group $\Gamma$. 
we may choose a $\delta\in \Gamma_n$ such that 
\[
F=\delta F;\qquad E=\delta E'.
\] 
Then the pull back $f\mid_k \delta$ 
%as in \eqref{eq:pb} 
is a formal Fourier-Jacobi series of weight $k$ for $F$ and the group $\delta^{-1} \Gamma\delta $.
%In particular, it has a formal Fourier expansion 
%\begin{align}
%\label{eq:fougen}
%(f\mid_k\delta) (\tau) = \sum_{T\in \Sym_{n}(\Q)} a_\delta(T)  \,e(\tr T \tau)
%\end{align}
%at the boundary component $F_0$.
It is easily seen that $f$ is  regular at the boundary component $E$ if and only if $f\mid_k\delta$ is regular at $E'$. 
In particular, $f$ is regular at all adjacent degree $0$ boundary components $E<F$ if and only if it regular for a set of representatives of the $\Gamma_F$-classes of such boundary components.
\end{remark}

In the special case when $F=F_m$ is equal to the standard boundary component of degree $m$ we identify $L_F^\vee$ with a sublattice of  $\Sym_{n-m}(\Q)$ via the symmetric bilinear form $(A,B)\mapsto \tr (AB)$. Then the action of $\gamma\in \Gamma_{F_m}$ as in \eqref{eq:GFm} on $T_2\in L_{F_m}^\vee$ is given by $p_\ell(\gamma)(T_2)= T_2[u]$. We may write the function $f_{T_2}$ 
%for $T_2\in L_{F_m}^\vee$ 
as 
\begin{align*}
f_{T_2}(\tau) = \phi_{T_2}(\tau_1,\tau_{12}) e(\tr T_2\tau_2),
\end{align*}
where $\phi_{T_2}$ is holomorphic on $\H_n$, independent of $\tau_2$, and satisfies the transformation law 
\begin{align}
\label{eq:tr2b}
\phi_{T_2}(\tau_1,\tau_{12}) e(\tr T_2\tau_2)\mid_k \gamma =\phi_{T_2[u]}(\tau_1,\tau_{12}) e(\tr T_2[u]\tau_2)
\end{align}
%\eqref{eq:tr2} 
for all $\gamma\in \Gamma_{F_m}$.
The above formal series \eqref{eq:fsgen} can be rewritten as
\begin{align}
\label{eq:ffj}
f(\tau) = \sum_{\substack{T_2\in L_{F_m}^\vee\\ T_2\geq 0}} \phi_{T_2}(\tau_1,\tau_{12}) \,e(\tr T_2 \tau_2),
\end{align}
where the coefficients $\phi_{T_2}$ are weakly holomorphic Jacobi forms of index $T_2$ and weight $k$ for the group $\Gamma_{F_m}^J$.
Under this identification the Fourier expansion 
%as in \eqref{eq:fj-fe1} 
of $f_{T_2}$ at the standard degree $0$ boundary component $F_0<F_m$  can be written as 
\begin{align}
\label{eq:fj-fe}
\phi_{T_2}(\tau_1,\tau_{12}) = \sum_{T_1 \in \Sym_{m}(\Q)} \sum_{T_{12}\in\Mat_{m,n-m}(\Q)} a\zxz{T_1}{T_{12}}{{}^tT_{12}}{T_2} e(\tr T_1 \tau_1 + 2\tr {}^tT_{12} \tau_{12}).
\end{align}
Putting these expansions together, we obtain the  formal Fourier expansion 
\begin{align}
\label{eq:fouf}
f(\tau) = \sum_{T\in \Sym_{n}(\Q)} a(T)  \,e(\tr T \tau)
\end{align}
at $F_0$. 
Moreover, $f$  
is regular at  $F_0$ if 
%its formal Fourier expansion \eqref{eq:fouf} defines a formal Fourier-Jacobi series of weight $k$ for the boundary component $F_0$ in the sense of Definition \ref{def:ffjF}. That is, 
$a(T)=0$ unless $T\geq 0$, and 
\begin{align}
\label{eq:sym}
\det(u)^k a(T)= a(T[u]) 
\end{align}
for all $T\in L_{F_0}^\vee\subset \Sym_{n}(\Q)$ and all $\zxz{u}{s}{0}{{}^tu^{-1}}\in \Gamma_{F_0}$.

%In particular, the Fourier coefficients $a_\delta(T)$ in the  expansion \eqref{eq:fougen} satisfy \eqref{eq:sym} for all  $\kzxz{u}{s}{0}{{}^tu^{-1}}\in (\delta^{-1} 
%\Gamma \delta)_{F_0}$ and all 
%$T\in \Sym_{n}(\Q)$. Note that this regularity condition is independent of the choice of $\delta$, but the coefficients $a_\delta(T)$ may depend on $\delta$.

\begin{comment}
----

Choose $\delta \in \Gamma_n$ such that $F = \delta F_m$. Then $\delta^{-1}G_F\delta= G_{F_m}$ and   $\delta^{-1}\Gamma_F\delta= (\delta^{-1}\Gamma\delta)_{F_m}$. Let $f$ be a formal Fourier-Jacobi series of weight $k$ for the boundary component $F$ and the group $\Gamma$. The pulled back formal series 
\begin{align*}
(f\mid_k \delta) (\tau)= \sum_{t\in L^\vee} (f_t\mid_k \delta)(\tau)
\end{align*}
defines a formal Fourier-Jacobi series of weight $k$ for the standard boundary component $F_m$ and the conjugate group $\delta^{-1} \Gamma\delta $.
We call $f$ {\em regular} at the boundary component $\delta F_0$ if $f\mid_k\delta$ is regular at the standard boundary component $F_0$.
% for all $\gamma\in \Gamma_{n,m}$.
In particular, the coefficients $a_\delta(T)$ in the formal Fourier expansion
\begin{align}
\label{eq:fougen}
(f\mid_k\delta) (\tau) = \sum_{T\in \Sym_{n}(\Q)} a_\delta(T)  \,e(\tr T \tau).
\end{align}
satisfy \eqref{eq:sym} for all  $\kzxz{u}{b}{0}{{}^tu^{-1}}\in (\delta^{-1} 
\Gamma \delta)_{F_0}$ and all 
$T\in \Sym_{n}(\Q)$. 
\end{comment}

Now assume that $F$, $F'$ are two rational boundary component of degree $m$ and that there exist a  rational boundary component $E$ of degree $0$ such that $F>E$ and $F'>E$. 
Let $f$, $f'$ be formal Fourier-Jacobi series of weight $k$ for the group $\Gamma$ and the boundary components $F$, $F'$, respectively.
Then $f$ and $f'$ have formal Fourier expansions 
\begin{align*}
f = \sum_{\substack{s\in L_E^\vee }} f_{s}, \qquad 
f' = \sum_{\substack{s\in L_E^\vee }} f'_{s}
\end{align*}
at $E$ as in \eqref{eq:fouf1}. We say that $f$ is {\em compatible} with $f'$ at $E$ if 
these formal expansions agree.

\begin{definition}
\label{def:symffj}
Let $I_m$ be the set of all rational boundary components of degree $m$.
A {\em formal Siegel modular form} of weight $k$ and cogenus $l$
 for the group  $\Gamma$ is a family $(f_F)_{F\in I_{n-l}}$, where $f_F$ is a  formal Fourier-Jacobi series of weight $k$ for the boundary component $F$ and the group $\Gamma$ satisfying the following conditions:
\begin{itemize}
\item[(i)] for all $F\in I_{n-l}$ and all $\gamma \in \Gamma$ we have 
$f_F\mid_k \gamma= f_{\gamma^{-1} F}$;
\item[(ii)] for all pairs $F,F'\in I_{n-l}$ and all degree $0$ boundary components $E\in I_0$ with  $F>E$ and $F'>E$, the formal Fourier-Jacobi series $f_F$ and $f_{F'}$ are compatible at $E$.
\end{itemize}
We write $
\operatorname{FM}^{(n,l)}_k(\Gamma)$ for the complex vector space of formal Siegel modular forms of weight $k$ and cogenus $l$ for $\Gamma$.
\end{definition}

%\texttt{Maybe better: Not take a system of representatives $I_m$. Define $f=(f_F)$ as a family indexed by all rational boundary components of degree $m$. Ask the additional condition $f_F = f_{F'}\mid \gamma$ for all $\gamma\in \Gamma$ and $F'=\gamma F$. Then the compatibility condition can be simplified to the case $E=E'$.}

\begin{remark}
\label{rem:regu}
1. Because of condition (i), the family $(f_F)_{F\in I_{n-l}}$ is determined by the $f_F$ for $F$ in a system of representatives for $I_{n-l}/\Gamma$.
According to Remark \ref{rem:class2} it suffices to check condition (ii) for a set of representatives of $\Gamma_F$-classes of adjacent degree $0$ rational boundary components $E<F$. %In this sense it is a finite condition. 

2. Conditions (i) and (ii) imply that $f_F$ is regular at all adjacent degree $0$ rational boundary components $E<F$. In fact, if $\gamma\in \Gamma_E$, then $E<\gamma^{-1} F$. Applying condition (ii) for $f_F$ and $f_{\gamma^{-1}F}$ we see that the transformation law \eqref{eq:sym1} holds.
\end{remark}

\begin{example}
1. Let $f\in M_k(\Gamma)$ be a holomorphic Siegel modular form of weight $k$ for $\Gamma$.
Then at each boundary component $F\in I_{n-l}$ the function $f$ has a (convergent) Fourier-Jacobi expansion $f_F$  as in \eqref{eq:fsgen}, which we may as well view as a formal series.
The family $(f_F)_{F\in I_{n-l}}$ defines a  formal Siegel modular form of weight $k$ and cogenus $l$ for $\Gamma$. 

2. Let $V\subset X_\Gamma^*$ be an open neighborhood of the boundary $\partial X_\Gamma^*$. Then any section $f\in \omega^{\otimes k}(V)$ defines an element of $\operatorname{FM}^{(n,l)}_k(\Gamma)$.
%a formal Siegel modular form of weight $k$ and cogenus $l$ for the group $\Gamma$. 

3.) Assume that $\Gamma=\Gamma_n$. Then any symmetric formal Fourier–Jacobi series of weight $k$ and cogenus $l$ in the sense of \cite{BR} defines an element of $\operatorname{FM}^{(n,l)}_k(\Gamma_n)$.
\end{example}

\subsection{Formal completion of the sheaf of modular forms} 
\label{sect:3.2}

Throughout this section we fix an arithmetic subgroup $\Gamma\subset \Symp_n(\Q)$. We briefly write $X$ for $X_\Gamma$ and $X^*$ for $X_\Gamma^*$.
We denote by 
\[
Y=\partial X^*\subset X^* 
\]
the boundary of $X$ as in \eqref{eq:boundary}.
%
%Let $F_{n-1}\subset \calD_n^*$ be the standard boundary component of genus $n-1$. Let $X_{\Gamma,F_{n-1}} = \bar \Gamma_{F_{n-1}}\bs F_{n-1}$ be the corresponding boundary stratum and denote by
%\[
%Y= \overline{X}_{\Gamma,F_{n-1}}\subset X
%\]
%the closure. 
Then $Y$ is a closed analytic subset of codimension $n$
%, which has a natural structure as a projective subvariety 
of the projective complex algebraic variety $X^*$.
If $n>1$ then $Y$ is connected.

\begin{comment}
The structure sheaf $\calO_{X^*}$ has the following description near the boundary. By possible conjugating by an element of $\Gamma_n$ it suffices to consider it on small neighborhoods of the standard boundary components $F_m$, where
$0\leq m< n$.
Let $\tau_1\in F_m$ and let $V\subset X^*$ be an open neighborhood of $\tau_1$ as in \eqref{eq:U}.
Then the sections  $\calO_{X^*}(V)$ of the structure sheaf over $V$ are given by all continuous functions 
\[
f: \tilde W_n(U,C)\to \C 
\]
which are holomorphic on $W_n(U,C)$ and satisfy $f\mid_0 \gamma =f$ for all $\gamma \in \Gamma_{\tau_1}$. 
\end{comment}

We let $\hat X^*=(\hat X^*, \calO_{\hat X^*})$ be the formal complex space given by the completion of $X^*$ along $Y$, and we write 
\begin{align}
\label{eq:complmor}
i:\hat X^*\to X^*
\end{align}
for the natural morphism of formal complex spaces. 
We denote by $\hat\omega^{\otimes k}$ the completion of the sheaf $\omega^{\otimes k}$ of modular forms of weight $k$ with respect $Y$. Since $\omega^{\otimes k}$ is coherent, the natural map $i^*( \omega^{\otimes k})\to \hat\omega^{\otimes k}$ is an isomorphism. 
The adjunction map $\omega^{\otimes k}\to i_*i^* \omega^{\otimes k}$ defines an injective map on global sections
\begin{align}
\label{eq:keymap}
\omega^{\otimes k}(X^*)\to \hat\omega^{\otimes k}(\hat X^*).
\end{align}

Let $W(F_m)$ be a connected open neighborhood of $F_m$ as in Proposition \ref{prop:boundary}, and put 
\begin{align*}
%\label{eq:U2}
V=\Gamma_{F_m}\bs W(F_m)\subset X^*.
\end{align*}
As before, we view $\Gamma_{F_m}'$ as a lattice $L_{F_m}$ in $\Sym_{n-m}(\Q)$, and identify its dual $L_{F_m}^\vee$ with a sublattice of the same space via the symmetric bilinear form $(A,B)\mapsto \tr (AB)$. 
Recall the description of $\omega^{\otimes k}$ on $V$ given in Lemma \ref{lem:sect1}. 
We now give a description of the completion $\hat\omega^{\otimes k}$ on this boundary stratum.

\begin{proposition}
\label{prop:compl}
The space $\hat\omega^{\otimes k}(V)$ of sections is given by the space of all
%is naturally identified with the space of 
{\em formal} series 
\begin{align}
\label{eq:partly}
f(\tau) = \sum_{\substack{T_2\in L_{F_m}^\vee\\ T_2\geq 0 }}
\phi_{T_2}(\tau_1,\tau_{12}) \,e(\tr T_2 \tau_2)
\end{align}
%\begin{align}
%\label{eq:partly}
%f(\tau) = \sum_{T_2^*\in \frac{1}{h}\Z_{\geq 0} } \bigg(\sum_{\substack{T_2\in \Sym_{n-m}(\Q)\\ T_2=\kzxz{T_2^1}{T_2^{21}}{{}^tT_2^{12}}{T_2^*}\geq 0}} \phi_{T_2}(\tau_1,\tau_{12}) \,e(\tr T_2^1 \tau_2^1+2\tr {}^t T^{12}_2 \tau_2^{12}) \bigg)e(T_2^*\tau_2^{2}),
%\end{align}
satisfying the following conditions:
\begin{itemize}
\item[(i)] The coefficients $\phi_{T_2}(\tau_1,\tau_{12})$ are holomorphic on $W(F_m)\cap \calD_n$ for all $T_2\in L_{F_m}^\vee$. 

\item[(ii)] The transformation law $f\mid_k \gamma = f$ holds for all $\gamma\in \Gamma_{F_m}$.
\item[(iii)] For all $S\in L_{F_m}$ which are primitive, positive semidefinite  of rank $1$, and for all $t\in \Z_{\geq 0}$, the sub-series 
\[
\sum_{\substack{T_2\in L_{F_m}^\vee\\ T_2\geq 0 \\ \tr(T_2 S)=t}}
\phi_{T_2}(\tau_1,\tau_{12}) \,e(\tr T_2 \tau_2)
\]
converges normally on $W(F_m)\cap \calD$ and defines a holomorphic function there. 
\end{itemize}
\end{proposition}

\begin{remark}
1. In condition  (ii) the transformation law is to be understood coefficientwise as in \eqref{eq:tr2b}.

2. In condition (iii) of Proposition \ref{prop:compl}, the matrix $S$ determines a $1$-dimensional rational isotropic subspace of $U(F_m)$, and therefore a degree $n-1$ rational boundary component $E$ with $E\geq F_m$.
% has $F_m$ as an adjacent boundary component.
Now condition (iii) means that in the formal Fourier-Jacobi expansion of $f$ with respect to $E$  all formal Fourier-Jacobi coefficients (of cogenus $1$)  converge.

3. In the case  $m=n-1$, the lattice $L_{F_m}$ has rank $1$, and hence $L_{F_m}=h\Z$ for some positive rational number $h$. Then $L_{F_m}^\vee=h^{-1}\Z$ and the conditions of Proposition \ref{prop:compl} mean that 
\begin{align}
\label{eq:exp2}
f(\tau) = \sum_{T_2\in \frac{1}{h}\Z_{\geq 0} } \phi_{T_2}(\tau_1,\tau_{12}) \,e( T_2 \tau_2),
\end{align}
is a formal series, whose coefficients $\phi_{T_2}(\tau_1,\tau_{12})$ are holomorphic on 
$\H_{n-1}\times \C^{n-1}$ and satisfy the transformation law (ii) of a Jacobi form for $\gamma\in \Gamma_{F_{n-1}}$.
\end{remark}

We now turn to the proof of Proposition \ref{prop:compl}, which will occupy the rest of this subsection.
%\subsection{Proof of Proposition \ref{prop:compl}}
%\label{sect:3.3}
We use the Grothendieck comparison theorem  \cite[Theorem 4.1.5]{EGA} in the category of formal complex analytic spaces \cite[Theorem 2]{Ba} to reduce the computation to a smooth toroidal compactification. 

\subsubsection{Toroidal compactification}

We begin by recalling some facts about toroidal compactifications. Our main references are \cite{AMRT} and \cite{Na}.

Let $\Sigma =(\Sigma_F)_F$ be a $\Gamma$-admissible collection of 
fans 
%rational polyhedral cone decompositions 
as in \cite{AMRT}, Definition~5.1 in Chapter~3. 
According to Theorem 5.2 of loc.~cit.\ there exist a toroidal compactification $X^{\tor}= X^{\tor}_\Sigma$ of $X=X_\Gamma$ associated with $\Sigma$. 
Throughout we assume that $\Sigma$ is smooth. Then $X^{\tor}$ is a compact Moishezon space (i.e.\ an algebraic space over $\C$)  which is smooth in the orbifold sense. It contains $X$ as a dense open subset and the complement is a divisor with normal crossings.
Moreover, there is a proper morphism 
\[
\pi: X^{\tor}\to X^*
\]
to the Baily-Borel compactification $X^*$ which restricts to the identity on $X$. We define the sheaf of weight $1$ modular forms on $X^{\tor}$ as the pull back $\omega^{\tor}= \pi^*(\omega)$. The local Koecher principle implies that 
$\pi_*(\omega^{\tor}) = \pi_*\pi^* (\omega) = \omega$.

Let $F$ be a rational boundary component of degree $m$. Let $W(F)$ be a connected open neighborhood of $F$ as in Proposition \ref{prop:boundary}, and put 
\begin{align*}
%\label{eq:U2}
V=\Gamma_{F}\bs W(F)\subset X^*.
\end{align*}
We write $\hat V=i^{-1}(V)$ for the inverse image of $V$ under the morphism $i:\hat X^*\to X^*$.
%\eqref{eq:complmor}. 
Moreover, we denote by $\hat X^{\tor}$ the completion of $X^{\tor}$ with respect to the toroidal boundary divisor $\pi^{-1}(Y)$, and by $\hat V^{\tor}$ the completion of $V^{\tor}=\pi^{-1}(V)$ with respect to $\pi^{-1}(Y\cap V)$. Then we have the commutative diagram of formal complex spaces
\[
\xymatrix{
\hat V^{\tor} \ar[r] \ar[d] & \hat X^{\tor} \ar[r]^{i^{\tor}} \ar[d]^{\hat \pi} & X^{\tor}\ar[d]^\pi\\
\hat V \ar[r] & \hat X^* \ar[r]^i & X^*},
\]
where the vertical morphisms are proper. We denote by $\hat\omega^{\tor}$ the completion of $\omega^{\tor}$  with respect to $\pi^{-1}(Y)$. Since $\omega^{\tor}$ is coherent, the natural map $(i^{\tor})^*( \omega^{\tor})\to \hat\omega^{\tor}$ is an isomorphism, see \cite[Lemma 2.3]{Ba}.

We now recall the description of $X^{\tor}$ near a rational boundary component $F$ of degree $m$.
Let $\calD(F)=G_{F,\C}'\cdot  \calD_n$ be the Siegel domain of the third kind associated with $F$, see \cite[Definition III.4.5]{AMRT}, and put $\calD(F)'=\calD(F)/G_{F,\C}'$. 
Then there is a two step holomorphic fibration 
\[
\calD(F)\to \calD(F)'\to F,
\]
which is equivariant for the action of $G_F\cdot G_{F,\C}'$. Here $\calD(F)\to \calD(F)'$ is a principal $G_{F,\C}'$-bundle, and $\calD_n\subset \calD(F)$ is an open subset.
If $F=F_m$ is the standard boundary component of degree $m$ we have 
\begin{align*}
\calD(F_m) &=  \left\{ \zxz{\tau_1}{\tau_{12}}{{}^t\tau_{12}}{\tau_2}\mid \;\tau_1\in \H_m, \;\tau_{12}\in \C^{m\times (n-m)},\; \tau_2\in \Sym_{n-m}(\C)\right\},
\end{align*}
and the map $\calD(F)\to \calD(F)'$ is the natural projection to the $\H_m\times \C^{m\times (n-m)}$ part.

The quotient 
\[
T_F= G_{F,\C}'/\Gamma_F'
\]
is a complex algebraic torus of rank $r=\frac{1}{2}(n-m)(n-m+1)$, and 
\begin{align}
\label{eq:fb1}
\Gamma_F'\bs\calD(F)\to \calD(F)'
\end{align}
is a principal $T_F$-bundle containing $\Gamma_F'\bs\calD_n$ as an open subset. The fan $\Sigma_F$ determines a torus embedding
\begin{align}
\label{eq:fb2}
T_F\to T_{F,\Sigma_F}.
\end{align}
The toroidal variety on the right hand side has an open covering by affine toric varieties 
\begin{align}
\label{eq:fb3}
T_F\to T_{F,\sigma}
\end{align}
for the cones $\sigma \in \Sigma_F$. Recall that if $N$ denotes the co-character lattice of $T_F$ and $N^\vee$ its dual (the character lattice), then $T_{F,\sigma} =\Spec \C[\sigma^\vee\cap N^\vee]$.
By taking the contraction product of \eqref{eq:fb1} and \eqref{eq:fb2}, we obtain the fiber bundle 
\begin{align}
\label{eq:fb4}
(\Gamma_F'\bs\calD(F))_{\Sigma_F}:=(\Gamma_F'\bs\calD(F))\times^{T_F}  T_{F,\Sigma_F}
\end{align}
over $\calD(F)'$ associated to \eqref{eq:fb1} with fiber $T_{F,\Sigma_F}$.
%which is a fiber bundle over $\calD(F)'$ with fiber $T_{F,\Sigma_F}$.
Now define 
\begin{align}
\label{eq:fb5}
(\Gamma_F'\bs\calD_n)_{\Sigma_F}
%:=(\Gamma_F'\bs\calD(F))\times^{T_F}  T_{F,\Sigma_F}.
\end{align}
as the interior of the closure of $\Gamma_F'\bs\calD_n$ in $(\Gamma_F'\bs\calD(F))\times^{T_F}  T_{F,\Sigma_F}$. It can be viewed as a partial compactification of $\Gamma_F'\bs\calD_n$ in the direction $F$. The group $\Gamma_F/\Gamma_F'$ acts properly discontinuously on $(\Gamma_F'\bs\calD_n)_{\Sigma_F}$, and there is a holomorphic map 
\[
 (\Gamma_F/\Gamma_F')\bs (\Gamma_F'\bs\calD_n)_{\Sigma_F}\to X^{\tor},
\]
which restricts to an isomorphism in a sufficiently small neighborhood of the $F$-stratum of the toroidal boundary, see \cite{AMRT}, p.~175. In particular, $V^{\tor}$ can be identified with an open subset of the left hand side as follows. Define
\begin{align*}
\big(\Gamma_F'\bs (W(F)\cap\calD_n) \big)_{\Sigma_F}
%:=(\Gamma_F'\bs\calD(F))\times^{T_F}  T_{F,\Sigma_F}.
\end{align*}
as the interior of the closure of $\Gamma_F'\bs (W(F)\cap\calD_n)$ in $(\Gamma_F'\bs\calD(F))\times^{T_F}  T_{F,\Sigma_F}$.
The group $\Gamma_F/\Gamma_F'$ acts properly discontinuously, and the holomorphic map 
\[
 (\Gamma_F/\Gamma_F')\bs \big(\Gamma_F'\bs  (W(F)\cap\calD_n)\big)_{\Sigma_F} \to X^{\tor},
\]
is an open immersion with image $V^{\tor}$. 

For the rest of this subsection we assume that $F=F_m$ is the standard boundary component of degree $m$. The space of sections $\omega^{\otimes k}(V^{\tor})$ is given by all continuous functions $f:W(F_m)\to \C$ that are holomorphic on $W(F_m)\cap \calD_n$ and satisfy $f\mid_k \gamma = f$ for all $\gamma\in \Gamma_{F_m}$.  Then $f$ has a Fourier-Jacobi expansion of the form 
\begin{align}
\label{eq:partly2}
f(\tau) = \sum_{\substack{T_2\in L_{F_m}^\vee\\ T_2\geq 0 }}
\phi_{T_2}(\tau_1,\tau_{12}) \,e(\tr T_2 \tau_2).
\end{align}
Let $\rho\in \Sigma_{F_m}$ be a ray (i.e.\ a cone of dimension $1$), and let $e_\rho\in L_{F_m}$ be the unique primitive ray generator. The ray determines a toroidal boundary divisor $D_\rho$, and we write $\mathcal{I}_\rho\subset \calO_{X^{\tor}}$ for the corresponding ideal sheaf. Then  $(\omega^{\otimes k}\otimes \mathcal{I}_\rho)(V^{\tor})$ is given by the subspace of those $f\in \omega^{\otimes k}(V^{\tor})$ whose  Fourier-Jacobi coefficients $\phi_{T_2}$ vanish identically  for all $T_2\in L_{F_m}^\vee$ with $(T_2,e_\rho)=0$.
%satisfy 
%\[
%\phi_{T_2} \equiv 0\quad \text{ for all $T_2\in L_{F_m}^\vee$ with $(T_2,e_\rho)=0$.}
%\]
The following proposition gives a description of the formal completion $(\hat \omega^{\tor})^{\otimes k}$ over  $V^{\tor}$.

\begin{proposition}
\label{prop:compl2}
The space $(\hat \omega^{\tor})^{\otimes k}(V^{\tor})$ of sections is given by the space of all \emph{formal} series as in \eqref{eq:partly2} satisfying the following conditions:
\begin{itemize}
\item[(i)] The coefficients $\phi_{T_2}(\tau_1,\tau_{12})$ are holomorphic on $W(F_m)\cap \calD_n$ for all $T_2\in L_{F_m}^\vee$. 
\item[(ii)] The transformation law $f\mid_k \gamma = f$ holds for all $\gamma\in \Gamma_{F_m}$ in the sense of \eqref{eq:tr2b}.
\item[(iii)] For all rays $\rho\in \Sigma_{F_m}$ with primitive ray generator $e_\rho\in L_{F_m}$ and for all $t\in \Z_{\geq 0}$,
%there exists an open neighborhood $N_\rho$ of the boundary divisor $D_\rho$ in $(\Gamma_{F_m}'\bs (W(F_m)\cap\calD_n))_{\Sigma_F}$ 
%such that for all $t\in \Z_{\geq 0}$ 
the sub-series 
\begin{align}
\label{eq:subs}
\sum_{\substack{T_2\in L_{F_m}^\vee\\ T_2\geq 0 \\ (T_2, e_\rho)=t}}
\phi_{T_2}(\tau_1,\tau_{12}) \,e(\tr T_2 \tau_2)
\end{align}
converges normally on $W(F_m)\cap \calD_n$ and defines a holomorphic function there. 
\end{itemize}
\end{proposition}

\begin{proof}
The result is a consequence of the fiber bundle structure \eqref{eq:fb4} and Proposition~\ref{prop:4hol} below on toroidal varieties.
\end{proof}

\begin{proof}[Proof of Proposition \ref{prop:compl}]
According to the Grothendieck comparison theorem \cite[Theorem 2]{Ba}, the natural map 
\begin{align}
\hat \omega = (\pi_* \omega^{\tor})\,\hat{} \to \hat\pi_*(\hat \omega^{\tor})
\end{align}
is an isomorphism of $\calO_{\hat X^*}$-modules. In particular, 
there is a natural isomorphism
\begin{align*}
\hat \omega^{\otimes k} (V) \cong (\hat \omega^{\tor})^{\otimes k}(V^{\tor}),
\end{align*}
and therefore we may compute the left hand side 
using Proposition \ref{prop:compl2}.
Since the first two conditions of Proposition \ref{prop:compl2} agree with the first two conditions of Proposition \ref{prop:compl} it remains to compare the corresponding third conditions. 

We first notice that any $S\in L_{F_m}\subset \Sym_{n-m}(\Q)$ which is primitive and positive semidefinite  of rank $1$ determines a $1$-dimensional rational subspace of the isotropic subspace $U(F_m)\cong \R^{n-m}$ as in \eqref{eq:UFm}. It is given as the orthogonal complement of $\ker(S)$ with respect to the standard scalar product on $\R^{n-m}$. Therefore $S$ determines a rational boundary component $E$ of degree  $n-1$ with $E\geq F_m$, and $\Z S=L_E\subset L_{F_m}$. Moreover, it determines a ray $\R_{\geq 0} S\in \Sigma_E$, and hence, according to \cite{AMRT} Definition~5.1 in Chapter~3, also a ray in $\Sigma_{F_m}$. 
Consequently, condition (iii) in Proposition \ref{prop:compl2} implies (iii) in Proposition~\ref{prop:compl}.

It remains to show the implication in the other direction. Assume that condition (iii) in Proposition \ref{prop:compl} holds. Let $\rho\in \Sigma_{F_m}$ be a ray whose primitive ray generator $e_\rho\in L_{F_m}$ has rank greater than $1$, and let $t\in \Z_{\geq 0}$. We need to prove that the 
series \eqref{eq:subs}
%sub-series in condition (iii) of Proposition~\ref{prop:compl2} 
converges.

\begin{comment}
We first consider the case that $e_\rho$ is given by a diagonal matrix $\operatorname{diag}(d_1,\dots,d_{n-m})$ with entries $d_i\in \Z_{\geq 0}$, where $d_i>0$ for at least two indices $i$.
Let $E_{ii}\in \Sym_{n-m}(\Q)$ be the symmetric matrix with entries all $0$ except for a $1$ at the position $(i,i)$. Then $\R_{\geq 0}E_{ii}$ determines a ray in $\Sigma_F$ corresponding to a rational boundary component of degree $n-1$. Let $c_{i}\in \Q_{>0}$ such that $c_i E_{ii}\in L_{F_m}\cap (\R_{\geq 0}E_{ii})$ is a primitive  ray generator.  
For $T_2\in L_{F_m}^\vee$ positive semidefinite we have 
\[
0\leq (T_2,e_\rho)= \sum_{i=1}^{n-m} d_i  \cdot (T_2,E_{ii}).
\]
Hence if $(T_2,e_\rho)=t$, then $0\leq  (T_2,c_i E_{ii}) \leq \frac{c_i}{d_i}t$ for all $i$ with $d_i>0$. Consequently, the series \eqref{eq:subs} is a sub-series of 
\[
\sum_{\substack{T_2\in L_{F_m}^\vee\\ T_2\geq 0 \\ (T_2,c_i E_{ii})\leq\frac{c_i}{d_i}t}}
\phi_{T_2}(\tau_1,\tau_{12}) \,e(\tr T_2 \tau_2).
\]
The latter series converges normally, since it is a finite sum of normally convergent series by condition (iii) of Proposition \ref{prop:compl}. Therefore,  \eqref{eq:subs} is also normally convergent.
\end{comment}

%If the ray generator $e_\rho$ is not necessarily diagonal, 
Let $\gamma\in \GL_{n-m}(\Q)$ such that $D:=\gamma^{-1} e_\rho {}^t\gamma^{-1}$ is diagonal. Without loss of generality, we may assume that $D=\operatorname{diag}(d_1,\dots,d_{n-m})$ with $d_i\in \Z_{\geq 0}$ and $d_1>0$. 
Let $E_{ii}\in \Sym_{n-m}(\Q)$ be the symmetric matrix with entries all $0$ except for a $1$ at the position $(i,i)$.
For  every positive semidefinite $T_2\in \Sym_{n-m}(\Q)$ we have 
\[
0\leq \sum_{i=1}^{n-m} d_i  \cdot (T_2,E_{ii})=(T_2,D) .
\]
Hence, $0\leq  d_1 (T_2, E_{11}) \leq (T_2,D)$.
In particular, we obtain 
%For every positive semidefinite $T_2\in L_{F_m}^\vee$ we have 
\[
0\leq d_1 ({}^t\gamma T_2 \gamma,E_{11})\leq ({}^t\gamma T_2 \gamma,D),  
\]
and thereby 
\[
0\leq d_1 ( T_2, \gamma E_{11}{}^t\gamma)\leq ( T_2, \gamma D{}^t\gamma)=  ( T_2, e_\rho). 
\]
Let $c\in \Q_{>0}$ such that 
\[
S:= c \cdot \gamma E_{11}{}^t\gamma\in L_{F_m}^\vee
\]
is primitive. By construction, $S$ is positive semidefinite of rank $1$, and 
\[
0\leq ( T_2, S)\leq \frac{c}{d_1}  ( T_2, e_\rho).
%=  \frac{c}{d_1}  t.
\]
Hence if $(T_2,e_\rho)=t$, then $0\leq (T_2,S) \leq \frac{c}{d_1}t$. 
Consequently, the series in \eqref{eq:subs} is a sub-series of 
\[
\sum_{\substack{T_2\in L_{F_m}^\vee\\ T_2\geq 0 \\ (T_2,S)\leq\frac{c}{d_1}t}}
\phi_{T_2}(\tau_1,\tau_{12}) \,e(\tr T_2 \tau_2).
\]
The latter series converges normally, since it is a finite sum of normally convergent series by condition (iii) of Proposition \ref{prop:compl}. Therefore,  \eqref{eq:subs} is also normally convergent.
%As before the latter series converges normally, and therefore \eqref{eq:subs} is also normally convergent.
This concludes the proof of the Proposition.
\end{proof}

\subsubsection{Completion of toroidal varieties}

Here we summarize some facts about the completion of toroidal varieties at their boundary. This can be reduced to considering affine toric varieties.

Let $r$ and $d$ be positive integers with $1\leq r\leq d$. 
Let $\calO_{\C^d}$ be the sheaf of holomorphic functions on $\C^d$ and consider the ideal sheaf in $I\subset\calO_{\C^d}$ generated by $(z_1\cdots z_r)$.

%
%
%Let $a=(a_1,\dots,a_d)\in \C^d$ and write 
%\[
%\C\{z-a\}=\C\{z_1-a_1,\dots ,z_d-a_d\}
%\] 
%for the $\C$-algebra of convergent complex power series in $z_1-a_1,\dots,z_d-a_d$. This is a local ring with maximal ideal given by the complex power series that vanish at the point $a$. The completion of $\C\{z-a\}$ with respect to the maximal ideal is the ring $\C[[z-a]]$ of formal power series.

\begin{proposition}
\label{prop:2hol}
Let $a=(a_1,\dots,a_d)\in \C^d$ with $a_1\cdots a_r=0$. Let $U\subset \C^d$ be an open polydisc of radius $R=(R_1,\dots,R_d)$ around $a$. Assume  that $R_i< |a_i|$ for all $i\in \{1,\dots,r\}$ with $a_i\neq 0$.
%Then we have a canonical isomorphism 
The family of natural maps of sheaves 
\[
\calO_{U}/(z_1\cdots z_r)^n\to  \calO_{U}/(z_1-a_1,\dots,z_d-a_d)^n
\]
for $n\in \Z_{>0}$ induces an injective map
\[
\left(\lim_n \calO_{U}/(z_1\cdots z_r)^n \right)(U) \to \C[[z-a]].
\]
Its image is given by those formal power series 
\begin{align}
\label{eq:hl3}
f=\sum_{\nu=(\nu_1,\dots,\nu_d)\in \N_0^d} c_\nu \cdot (z-a)^\nu
\end{align}
for which the subseries
\begin{align}
\label{eq:hl4}
f_{i,t} = \sum_{\substack{\nu=(\nu_1,\dots,\nu_d)\in \N_0^d\\ \nu_i=t}} c_\nu \cdot (z-a)^\nu
\end{align}
converge normally on $U$ for all $i\in J_a:=\{i\in \N\mid 1\leq i\leq r,\; a_i= 0\}$
%\{1,\dots,r\}$ with $a_i= 0$ 
and for all $t\in \N_0$, and hence define holomorphic functions there.
\end{proposition}

\begin{proof}
By the hypothesis on $a$ the set $J_a$
%:=\{i\in \N\mid 1\leq i\leq r,\; a_i= 0\}$ 
is non-empty.
For $n\in \N$, consider the morphisms of $\calO_U$-modules 
\begin{align}
\label{eq:hl1}
\xymatrix{
\displaystyle\calO_U/(z_1\cdots z_r)^n \ar[r]^\varphi & \displaystyle\prod_{\substack{i\in J_a}} \calO_U/(z_i)^n \ar[r]^-{\pi_j} & \displaystyle\calO_U/((z_1-a_1)^n,\dots, (z_d-a_d)^n).
}
\end{align}
Here $\varphi$ is induced by the quotient maps, and the $\pi_j$ are obtained by composing the $j$-th projection with the quotient map for $j\in J_a$.
It is easily checked on stalks that $\varphi$ is injective. Notice that for $i\in \{1,\dots,r\}$ with $a_i\neq 0$  the class of $z_i$ is invertible in $\calO_{U,b}$ for all $b\in U$ by the assumption on $U$.
%The condition on $R$ implies that for all points $b=(b_1,\dots,b_n)\in U$ the vanishing of $b_i$ implies that $a_i=0$.
%
By taking the limit over $n$ and sections over $U$, we obtain
\begin{align}
\label{eq:hl2}
\xymatrix{
\displaystyle \big(\lim_n \calO_U/(z_1\cdots z_r)^n\big)(U) \ar[r]^{\varphi_U} & \displaystyle\prod_{\substack{i\in J_a}} \big(\lim_n \calO_U/(z_i)^n\big)(U) \ar[r]^-{\pi_{j,U}} & \displaystyle \C[[z-a]].
}
\end{align}
Since the limit and the global sections functors are left exact, the map $\varphi_U$ is injective. 
\begin{comment}
On the right hand side of \eqref{eq:hl2} we have identified
\[
\big(\lim_n \calO_U/((z_1-a_1)^n,\dots, (z_d-a_d)^n)\big)(U)\cong 
\big(\lim_n \calO_U/(z_1-a_1,\dots, z_d-a_d)^n\big)(U)\cong \C[[z-a]].
\]
\end{comment}

Let $i\in J_a$. Since $U$ is a Stein domain, by taking power series expansions, the limit
\[
\big(\lim_n \calO_U/(z_i)^n\big)(U)
%\cong \lim_n \big(\calO_U/(z_i)^n\big)(U)
\cong \lim_n \calO_U(U)/(z_i)^n
\]
can be identified with the subalgebra of $\C[[z-a]]$ consisting of those formal power series 
as in \eqref{eq:hl3} for which the subseries
\begin{align*}
f_{i,t} = \sum_{\substack{\nu=(\nu_1,\dots,\nu_d)\in \N_0^d\\ \nu_i=t}} c_\nu \cdot (z-a)^\nu
\end{align*}
converge on $U$ for all 
%$i$ with $a_i=0$ and all 
$t\in \N_0$,
and hence define holomorphic functions there. In particular,  the maps $\pi_{j,U}$ are injective. Moreover,  the composition $\pi_{j,U}\circ \varphi_U$ is independent of $j$. Therefore the image $\Im(\varphi_U)$ of the map $\varphi_U$ is contained in the space $B$ of formal power series as claimed in the statement of the proposition. 
%We leave the proof of the other inclusion $B\subset \Im(\varphi_U)$ to the reader.

To prove the inclusion 
$B\subset\Im(\varphi_U)$, let $f\in B$ and denote its power series expansion as in \eqref{eq:hl3}. Let $n\in \N$. The convergence of the subseries in \eqref{eq:hl4} for all 
$i\in J_a$ and all $t\in \N_0$ implies that 
\[
F_n:= \sum_{\substack{\nu\in \N_0^d\\ \exists i\in J_a:\; \nu_i<n}} c_\nu\cdot  (z-a)^\nu 
\]
converges and hence defines an element of $\calO_U(U)$. 
Moreover, the image $F_n+(z_1\cdots z_r)^n\in \calO_U(U) /(z_1\cdots z_r)^n$
is mapped to $F_{n-1}+(z_1\cdots z_r)^{n-1}\in \calO_U(U) /(z_1\cdots z_r)^{n-1}$ under the natural projection.  Consequently, the family $F=(F_n)_{n\in \N}$ defines an element of the projective limit $\lim _n \calO_U(U)/ (z_1\cdots z_r)^n$ with the property that 
$\varphi_U (F)= f$.
This concludes the proof of the proposition. 
\end{proof}

\begin{comment}
\begin{remark}
\label{rem:2hol}
The stalk $\left(\lim_n \calO_{U}/(z_1\cdots z_r)^n \right)_a$ is given by those formal power series $f$ as in \eqref{eq:hl3}
in $\C[[z-a]]$ for which there exists a neighborhood $V$ of $a$ on which the subseries $f_{i,t}$ 
converge normally for all $i\in J_a$ and for all $t\in \N_0$.
\end{remark}
\end{comment}

Let $N\cong \Z^d$ be a lattice with dual $M=N^\vee$. Write $\langle n,m \rangle$ for the natural $\Z$-bilinear pairing of $n\in N$ and $m\in N^\vee$.
Recall that a {\em rational polyhedral cone} $\sigma$ in $N_{\mathbb{R}}$ is a subset which is generated over $\mathbb{R}_{\geq 0}$ by finitely many vectors $n_1,\dots,n_l\in N$, that is, 
\[
\sigma=\{ t_1n_1+\dots + t_ln_l\mid \; t_1,\dots t_l\in \R_{\geq 0} \}.
\]
The \emph{dual cone} of $\sigma$ is defined by  
\[
\sigma^{\vee}=\big\{m\in N^{\vee}_{\mathbb{R}}\mid \; \text{$\langle n,m\rangle\geq 0$ for all $n\in\sigma$}\big\}.
\] 
It is a rational polyhedral cone in $N^{\vee}_{\mathbb{R}}$. 
The cone $\sigma$ is called \emph{strongly convex} if it contains no full lines. This is equivalent to $\sigma^{\vee}$ being of full dimension $d$ in $N^{\vee}_{\mathbb{R}}$. 
%For a cone $\sigma$ in $N_\R$ we also define 
%\[
%\sigma^\perp = \{m\in N^\vee_\R\mid \; \text{$\langle n,m\rangle =0$ for all $n\in \sigma$}\}.
%\]
%
For any rational polyhedral  cone $\sigma$, the monoid $\sigma^{\vee}\cap N^{\vee}$ is finitely generated, and the monoid algebra $A_\sigma=\C[\sigma^\vee\cap N^\vee]$ is a finitely generated $\C$-algebra. The affine toric variety 
\[
T_{\sigma}=\operatorname{Spec} \C[\sigma^\vee\cap N^\vee]
\] 
associated with $\sigma$ is a normal irreducible affine algebraic variety over $\C$, containing the torus $T_{N}=
\operatorname{Spec} \C[N^\vee]$ as an open dense subvariety.

Let $\Sigma$ be a rational fan in $N_\R$, that is, a collection of strongly convex rational polyhedral cones in $N_\R$ satisfying the conditions:
\begin{itemize}
\item If $\sigma \in \Sigma$ and $\tau\leq \sigma$ is a face of $\sigma$, then $\tau\in \Sigma$.
\item If $\sigma \in \Sigma$ and $\tau\in \Sigma$, then $\sigma\cap \tau\leq \sigma$ and $\sigma\cap \tau\leq \sigma$.
\end{itemize}
We write $T_\Sigma$ for the toric variety over $\C$ associated with $\Sigma$. It is obtained by gluing the $T_\sigma$ for $\sigma\in \Sigma$ along common faces, see e.g.~\cite{Fu}. 

The dimension of a cone $\sigma$ in $N_\R$ is the dimension of the $\R$-vector space $\sigma+(-\sigma)$. A rational polyhedral cone is called {\em smooth} if it is generated by part of a basis of $N$. In this case it has exactly $l=\dim\sigma$ edges (i.e.~$1$-dimensional faces) and a canonical set of generators (which is part of a basis of $N$) is given by the ray generators of the edges of $\sigma$.
%This is the case if and only if the ray generators of the edges (i.e.~the $1$-dimensional faces) of $\sigma$ are part of a basis of $N$.
A fan $\Sigma$ is called smooth if all its cones are smooth. This is equivalent to $T_\Sigma$ being smooth.

We now consider the toric variety $X_\Sigma^{\operatorname{an}}$ associated with a smooth rational fan $\Sigma$ in $N_\R$ as a complex analytic space.
Let $\hat{X}_\Sigma^{\mathrm{an}}$ be the completion of $X_\Sigma^{\mathrm{an}}$ at the  toric boundary divisor 
\[
D_\Sigma = \sum_{\substack{\tau\in \Sigma\\ \dim \tau =1}} D_\tau,
\]
where $D_\tau$ denotes the toric divisor associated with a ray $\tau$. We write $e_\tau\in N\cap\tau$ for the ray generator of $\tau$.
% as in \cite{Ba}. 
Proposition \ref{prop:2hol} immediately implies the following result.

\begin{proposition}
\label{prop:4hol}
Let $a$ be a point in the torus orbit corresponding to a cone $\sigma\in \Sigma$, and let $U\subset T_\sigma^{\mathrm{an}}$ be an open polydisc 
around $a$. 
The $\C$-algebra $\calO_{\hat X_\Sigma^{\operatorname{an}}}(U)$ is given by those formal power series  
\[
f=\sum_{\substack{\nu\in N^\vee}} c_\nu \cdot z^\nu\in \C[[N^\vee]]
\]
with the following properties:
\begin{itemize}
\item[(i)] We have $c_\nu = 0$ unless $\nu\in \sigma^\vee$.
% for all cones $\sigma\in \Sigma$.
\item[(ii)] 
For all edges $\tau\leq \sigma$, and for all $t\in \N_0$ the subseries 
\[
f_{\tau,t}=\sum_{\substack{\nu\in N^\vee\\ \langle e_\tau,\nu\rangle =t}} c_\nu \cdot z^\nu
\]
converges normally and defines a holomorphic function on $U$. 
%some open neighborhood of $D_\tau$.
% Here $e_\tau\in N\cap\tau$ denotes a ray generator of $\tau$.
\end{itemize}
Here we have put $z^\nu= e^{2\pi i \langle \zeta,\nu\rangle}$ for $\zeta\in N_\C/N$.
\end{proposition}

\begin{proof}
Denote by $l$ the dimension of $\sigma$.
Let $\tau_1,\dots,\tau_l$ be the edges of $\sigma$ and let $e_i\in \tau_i\cap N$ be the corresponding ray generators. Then $e_1,\dots, e_l$ is the minimal set of generators of $\sigma$ and there exist $e_{l+1},\dots ,e_d\in N$ such that $e_1,\dots,e_d$ are a lattice basis of $N$. Let $e_1^\vee,\dots, e_d^\vee\in N^\vee$ be the corresponding dual basis of $N^\vee$. 
The assignment $e_i^\vee\mapsto X_i$ determines a $\C$-algebra isomorphism
\begin{align}
\label{eq:iso}
A_\sigma=\C[\sigma^\vee\cap N^\vee]\to \C[X_{l+1}^{\pm 1},\dots , X_d^{\pm 1}][X_1,\dots,X_l], 
\end{align}
and hence an isomorphism $(\C^\times)^{d-l}\times \C^l\cong T_\sigma^{\mathrm{an}}$. The vanishing ideal of the restriction of $D_\Sigma$ to $T_\sigma$ is given by 
\[
I_\sigma= (e_1^\vee+\dots +e_l^\vee)\subset \C[\sigma^\vee\cap N^\vee].
\]
It corresponds to the ideal 
$(X_1\cdots X_l)=(X_1\cdots X_d)\subset \C[X_{l+1}^{\pm 1},\dots , X_d^{\pm 1}][X_1,\dots,X_l]$ under \eqref{eq:iso}.
Hence we may use Proposition \ref{prop:2hol} to compute the completion 
\[
\hat\calO_{T_\sigma^{\operatorname{an}}, I_\sigma} = \lim_n \calO_{X_\sigma^{\operatorname{an}}}/I_\sigma^n 
\]
of the structure sheaf $\calO_{T_\sigma^{\operatorname{an}}}$ at the ideal $I_\sigma$.
With the above identifications we obtain the assertion.
\end{proof}

%\begin{corollary}
%\label{prop:4hol}
%Let $a$ be a point in the torus orbit corresponding to a cone $\sigma\in \Sigma$, and let $U\subset T_\sigma^{\mathrm{an}}$ be an open polydisc 
%around $a$. 
%The $\C$-algebra $\calO_{\hat X_\Sigma^{\operatorname{an}}}(U)$ is given by those formal power series  
%\[
%f=\sum_{\substack{\nu\in N^\vee}} c_\nu \cdot z^\nu\in \C[[N^\vee]]
%\]
%with the following properties:
%\begin{itemize}
%\item[(i)] We have $c_\nu = 0$ unless $\nu\in \sigma^\vee$.
%% for all cones $\sigma\in \Sigma$.
%\item[(ii)] 
%For all edges $\tau\leq \sigma$, and for all $t\in \N_0$ the subseries 
%\[
%f_{\tau,t}=\sum_{\substack{\nu\in N^\vee\\ \langle e_\tau,\nu\rangle =t}} c_\nu \cdot z^\nu
%\]
%converges and defines a holomorphic function on $U$. 
%% Here $e_\tau\in N\cap\tau$ denotes a ray generator of $\tau$.
%\end{itemize}
%\end{corollary}

\subsection{Formal Siegel modular forms of cogenus $1$}
\label{sect:3.3}

Let $f\in \hat \omega^{\otimes k}(\hat X^*)$ be a global section.
\begin{comment}
Restricting $f$ to 
$\Gamma_{F_{n-1}}\bs W(F_{n-1})$ and using Proposition \ref{prop:compl}, we see that $f$ has a formal Fourier-Jacobi expansion
as in \eqref{eq:exp2}. On the other hand, restricting $f$ to  $\Gamma_{F_{0}}\bs W(F_{0})$, we find that $f$ has a formal Fourier expansion 
\begin{align*}
f(\tau) = \sum_{\substack{T\in L_{F_0}^\vee\\ T\geq 0}} a(T)  \,e(\tr T \tau),
\end{align*}
where the coefficients $a(T)$ satisfy 
\[
a(T[u] )=  \det(u)^k\cdot a(T)
\]
for all $u\in \GL_n(\Z)$ such that \eqref{eq:matg22} belongs to $\Gamma_{F_0}$. Comparing the two expansions on the (non-trivial) intersection of the  two open sets, we obtain that 
\[
\phi_{T_2}(\tau_1,\tau_{12}) = \sum_{T_1 \in \Sym_{n-1}(\Q)} \sum_{T_2\in\Q^{1\times(n-1)}} a\zxz{T_1}{T_{12}}{{}^tT_{12}}{T_2} e(\tr T_1 \tau_1 + 2\tr {}^tT_{12} \tau_{12}).
\]
In particular, $f$ is regular at the standard boundary component $F_0$. Repeating the same argument with 
$f\mid_k \delta$, where $\delta\in \Gamma_{n,n-1}$, we see that 
$f$ is regular at the boundary component $\delta F_0$ adjacent to $F_{n-1}$.
\end{comment}
If $F$ is any rational boundary component of degree $n-1$, we may restrict $f$ to $\Gamma_{F}\bs W(F)$. By Proposition \ref{prop:compl}, we obtain a formal Fourier-Jacobi series $f_F$ of weight $k$ for $F$ and the group $\Gamma$. In particular, for the standard boundary component $F_{n-1}$ we obtain an expansion as in \eqref{eq:exp2}.
%It is regular at all adjacent degree $0$ boundary components, and 
The family $(f_F)_{F\in I_{n-1}}$ of formal Fourier-Jacobi series is compatible in the sense of Definition 
\ref{def:symffj}. Hence it determines a formal Siegel modular form of weight $k$ and cogenus $1$ for $\Gamma$. The assignment $f\mapsto (f_F)$ defines a homomorphism of complex vector spaces
\begin{align}
\label{eq:mapffj}
\hat\omega^{\otimes k}(\hat X^*)\to \operatorname{FM}^{(n,1)}_k(\Gamma).
\end{align}
Since any global section $f\in \hat\omega^{\otimes k}(\hat X^*)$ is uniquely determined by  its formal Fourier-Jacobi expansion of cogenus $1$, the map is injective.

\begin{theorem}
\label{prop:ffj}
The above map 
%$\hat\omega^{\otimes k}(\hat X^*)\to \operatorname{FM}^{(n,1)}_k(\Gamma)$ 
in  \eqref{eq:mapffj} is an isomorphism.
\end{theorem}

\begin{proof}
We have to show that the map is surjective. Let $g=(g_F)_{F\in I_{n-1}}\in \operatorname{FM}^{(n,1)}_k(\Gamma)$. 
For every proper rational boundary component $F$ we chose a sufficiently small open neighborhood of $V(F) = \Gamma_{F}\bs W(F)$ as in Proposition \ref{prop:boundary}.
We show that $g$ determines a section in $\hat \omega^{\otimes k}(V(F))$ for every $F$ and that these sections agree on all pairwise intersections of these open neighborhoods. Hence they glue to a global section $f\in \hat \omega^{\otimes k}(\hat X^*)$ which maps to $g$ under the map  \eqref{eq:mapffj}.
First, we assume that $\Gamma=\Gamma_n(N)$  is the principal congruence subgroup of level $N$ and genus $n$.

1. We begin with the rank $n-1$ standard boundary component $F_{n-1}$.
According to Proposition \ref{prop:compl}, the expansion 
\begin{align}
\label{eq:ffj2}
g_{F_{n-1}}(\tau) = \sum_{\substack{T_2\in \Q}} \phi_{T_2}(\tau_1,\tau_{12}) \,e( T_2 \tau_2)
\end{align}
of $g_{F_{n-1}}$ 
as in \eqref{eq:ffj},
%\eqref{eq:fsgen} 
together with the transformation law \eqref{eq:tr2b}, imply 
that $g_{F_{n-1}}$ defines an element of $\hat \omega^{\otimes k}(V(F_{n-1}))$.

2. Now we consider the rank $0$ standard boundary component $F_{0}$.
The formal Fourier-Jacobi series $g_{F_{n-1}}$ has a formal Fourier expansion
\begin{align}
\label{eq:fouf2}
g_{F_{n-1}} = \sum_{T\in \Sym_{n}(\Q)} a(T)  \,e(\tr T \tau)
\end{align}
as in \eqref{eq:fouf}. Since $g_{F_{n-1}}$ is regular at the boundary component $F_0$ (by Remark \ref{rem:regu}), the Fourier coefficients $a(T)$ satisfy the transformation law \eqref{eq:sym}. Hence the series \eqref{eq:fouf2} 
%formal Fourier expansion of 
%$g_{F_{n-1}}$ 
satisfies conditions (i) and (ii) of Proposition~\ref{prop:compl} for $V(F_0)$. It remains to show that it also satisfies condition (iii).

To this end let $S\in L_{F_0}\subset \Sym_n(\Q)$ be primitive and positive semi-definite of rank $1$. 
Denote by $R\in L_{F_0}$ the positive semi-definite generator of the $1$-dimensional lattice $L_{F_{n-1}}$. Then $R$ is a matrix with entries all $0$ except for the position $(n,n)$.
%Denote by $E_{ii}\in \Sym_{n}(\Q)$ be the symmetric matrix with entries all $0$ except for a $1$ at the position $(i,i)$. Then $L_{F_{n-1}}\otimes \Q=\Q E_{nn}$.
%We denote by $R\in \Q E_{nn}$ a positive semi-definite generator of $L_{F_{n-1}}$.
%
Write 
\[
S= h\cdot u R \,{}^t u
\]
with $u\in \Gl_n(\Z)$ and $h\in \Q_{>0}$.
Consider the boundary component $F=\delta F_{n-1}$, where $\delta = \kzxz{u}{0}{0}{{}^tu^{-1}}\in \Gamma_{n,0}$, and the corresponding formal Fourier-Jacobi series $g_{F}$. 
The formal Fourier expansions of $g_{F_{n-1}}$ and $g_F$  at $F_0$ are given by 
\begin{align*}
g_{F_{n-1}} &= \sum_{t\geq 0} \sum_{\substack{T\geq 0 \\ \tr TR=t}} a(T) q^T,\\
g_{F} &= \sum_{t\geq 0} \sum_{\substack{T\geq 0 \\ \tr TS=t}} b(T) q^T.\\
\end{align*}
Since $F_{n-1}>F_0$ and $F>F_0$, the series  $g_{F_{n-1}}$ and $g_F$ must be compatible at $F_0$, and therefore $a(T)=b(T)$.
%
%We denote the coefficients in the formal Fourier expansion \eqref{eq:fougen} of $g_F\mid_k \delta'$ by  $a'(T)$ and compare them to the coefficients $a(T)$ above.
%%in \eqref{eq:fouf2}.
%The compatibility condition  \eqref{eq:comp2} means that
%\begin{align}
%\label{eq:comp3}
%a(T[u^{-1}]) = \det(u)^k a'(T)
%\end{align}
%(where we have used $\delta=1$ and $\beta = \delta'$).
Since $g_F$ is a formal Fourier-Jacobi series for the boundary component $F$, the formal Fourier-Jacobi coefficients 
\[
\sum_{\substack{T\geq 0 \\ \tr TS=t}} a(T) q^T
\]
converge normally for all $t$. 
%By \eqref{eq:comp3} we have 
%\begin{align*}
%\phi'_t(\tau_1,\tau_{12}) & =  \det(u)^k \sum_{\substack{T\geq 0 \\ \tr T \kzxz{0 }{0}{0}{h}=t}} a(T[u^{-1}]) e(\tr T\tau)\\
%& = \det(u)^k \sum_{\substack{T\geq 0 \\ \tr T S=t}} a(T) e(\tr T\tau[{}^t u]).
%\end{align*}
This gives the desired convergence condition (iii) of Proposition~\ref{prop:compl}, and therefore shows that the formal Fourier expansion of 
$g_{F_{n-1}}$ defines and element of $\hat \omega^{\otimes k}(V(F_{0}))$, which agrees with $g_{F_{n-1}}$ on $V(F_{0})\cap V(F_{n-1})$.

3. Now let $0\leq m\leq n-1$ and consider the boundary component $F_m$.
We use the formal Fourier expansion \eqref{eq:fouf2} of $g_{F_{n-1}}$ to define Fourier-Jacobi coefficients 
\[
\psi_{T_2}(\tau_1,\tau_{12})= \sum_{T=\kzxz{T_1}{T_{12}}{{}^t T_{12}}{T_2}\geq 0}
a(T) e(\tr T_1 \tau_1+2\tr T_{12}{}^t \tau_{12})
\]
for $T_2\in L_{F_m}^\vee\subset \Sym_{n-m}(\Q)$. 
These series converge normally, since they are sub-series of the Fourier expansions of the Fourier-Jacobi coefficients $\phi_{T_2^*}$ in \eqref{eq:ffj2}, where $T_2^*$ denotes the lower right entry of $T_2$.
Then we have the identity of formal series
\begin{align}
\label{eq:ffj3}
g_{F_{n-1}} = \sum_{\substack{T_2\in \Sym_{n-m}(\Q)}} \psi_{T_2}(\tau_1,\tau_{12}) \,e(\tr T_2 \tau_2).
\end{align}
By Lemma \ref{ass:gen} below, the transformation law of \eqref{eq:ffj2} under $\Gamma_{F_{n-1}}$, and  the transformation law of \eqref{eq:fouf2} under $\Gamma_{F_{0}}$ imply that \eqref{eq:ffj3} has the transformation law of Proposition \ref{prop:compl} (ii) 
for $\Gamma_{F_m}$. Moreover, condition (iii) of Proposition \ref{prop:compl} follows from the corresponding condition for the formal Fourier expansion \eqref{eq:fouf2}. Hence \eqref{eq:ffj3} defines a section in $\hat \omega^{\otimes k}(V(F_{m}))$, which is compatible with the sections in $\hat \omega^{\otimes k}(V(F_{n-1}))$ and $\hat \omega^{\otimes k}(V(F_{0}))$ that were constructed before.
 
4. Let $E$ be any proper rational boundary component. Let $0\leq m\leq n-1$ and $\delta\in \Gamma_n$ such that $E=\delta F_m$. Consider the translated series $g_F\mid_k \delta$, which defines a formal Fourier Jacobi series for $F_m$ and the group $\delta^{-1}\Gamma \delta$ by Remark \ref{rem:class0}. According to part (3.) above it defines a section 
$\tilde g \in \hat \omega^{\otimes k}(\delta^{-1}\Gamma_E\delta \bs \delta^{-1} W(E))$.
By taking its pullback $\tilde g \mid_k \delta^{-1}$, we obtain a section in 
$\hat\omega^{\otimes k}(V(E))$.

5. We now show that the local sections of 1.--4.\ agree on all open intersections.
Since the support of $\hat\omega^{\otimes k}$ is given by the boundary $Y$, it suffices to consider all pairs of proper rational boundary components $F$ and $F'$ for which 
\[
\bar F \cap \bar F ' \neq \emptyset.
\]
This implies that there exists a rational degree $0$ boundary component $E$ such that $F \geq E$ and $F'\geq E$. In view of part (3.)\ we may assume without loss of generality that $F$ and $F'$ both have degree $n-1$. We have to show that the restrictions to $V(E)$ of the sections of $\hat\omega^{\otimes k}(V(F))$ and  $\hat\omega^{\otimes k}(V(F'))$ constructed in (4.)\ agree, which is a consequence of the compatibility of $g_F$ and $g_{F'}$ at $E$. 
This concludes the proof if $\Gamma=\Gamma_n(N)$.

6. In the above argument we have only used the assumption that $\Gamma=\Gamma_n(N)$ in step (3.)\ for $0 < m< n-1$ when we invoked Lemma \ref{ass:gen}. If $n\leq 2$ there are no such integers $m$ and hence the argument also applies for any arithmetic subgroup $\Gamma\subset \Symp_n(\Q)$. 
Finally, assume that $n\geq 2$ and that $\Gamma\subset \Symp_n(\Q)$ is an arbitrary arithmetic subgroup. There exists a positive integer $N$ such that $\Gamma':=\Gamma_n(N)$ is a normal subgroup of $\Gamma$ of finite index. This corresponds to a finite covering $\pi: X_{\Gamma'}^*\to X_\Gamma^*$ of complex spaces and of their completions at the Satake boundaries. We may view the formal Siegel modular form $g$ as an element $\operatorname{FM}^{(n,1)}_k(\Gamma')$. By the above argument, it determines a section $f\in \hat \omega^{\otimes k}(\hat X^*_{\Gamma'})$, which is invariant under the action of $\Gamma/\Gamma'$. It descends to a section in $\omega^{\otimes k}(\hat X^*_{\Gamma})$, which maps to $g$ under the map \eqref{eq:mapffj}. 
\end{proof}

We conclude this subsection with the Lemma that was used in the proof of the above Theorem.

\begin{lemma}
\label{ass:gen}
Assume that $\Gamma=\Gamma_n(N)$  is the principal congruence subgroup of level $N$ and genus $n$. Then for every triple of adjacent rational boundary components $F>F'>F''$, where $F$ has degree $n-1$ and $F''$ has degree $0$, the stabilizer $\Gamma_{F'}$  is generated by 
$\Gamma_{F}\cap \Gamma_{F'}$ and $\Gamma_{F''}\cap \Gamma_{F'}$.  
\end{lemma}

\begin{proof}
Since $\Gamma=\Gamma_n(N)$ is a normal subgroup of the full Siegel modular group $\Gamma_n$, and since $\Gamma_n$ acts transitively on chains of adjacent rational boundary components of fixed degrees, it suffices to prove the assertion for triples of standard boundary components 
\[
F_{n-1} > F_m > F_0,
\]
where $n-1> m > 0$.
Let 
\[
\gamma=\begin{pmatrix}
a & 0 & * & *\\
* & u & * & *\\
c & 0 & d & *\\
0 & 0 & 0 & {}^tu^{-1}
\end{pmatrix}
\]
be an element of $\Gamma_{F_m}$, where $u\in \GL_{n-m}(\Z)$ is congruent to $1$ modulo $N$. Then 
\[
\delta=\begin{pmatrix}
1 & 0 & 0 & 0\\
0 & u & 0 & 0\\
0 & 0 & 1 & 0\\
0 & 0 & 0 & {}^tu^{-1}
\end{pmatrix}
\]
belongs to $\Gamma_{F_0}\cap \Gamma_{F_m}$. The product $\delta^{-1}\gamma$ is contained in 
  $\Gamma_{F_{n-1}}\cap \Gamma_{F_m}$, proving the assertion.
\end{proof}

\section{Modularity of formal Siegel modular forms}
\label{sect:4}

\subsection{Affine covering numbers}
\label{sect:4.1}
%Here we give some upper bounds on affine covering numbers of Siegel modular varieties.

Let $S$ be a scheme. Recall that the {\em affine covering number} $\acn(S)$ of  $S$ is defined as  one less than the smallest number of open affine sets required to cover $S$, see e.g.~\cite{At}, \cite{RV}. It gives an upper bound for the cohomological dimension of $S$, which is the largest
integer $j$ such that $H^j(S, \calF ) \neq  0$ for some quasicoherent sheaf $\calF$ \cite[Proposition 4.12]{RV}. 
If $S$ is quasi-projective, then according to \cite[Example 4.8]{RV}, we have the trivial bound  $\acn(S)\leq \dim(S)$.

We consider this notion for the Siegel modular variety $X_\Gamma= \Gamma\bs \H_n$ associated with an arithmetic subgroup $\Gamma\subset \Symp_n(\Q)$. 
Since the line bundle of modular forms on the Baily-Borel compactification $X_\Gamma^*$ is ample, the complement of the divisor $\dv(F)$ of any holomorphic modular form $F$ of weight $k$ defines an affine open subset of $X_\Gamma^*$. If $F$ is a cusp form, then $X_\Gamma^*\setminus  \dv(F)$ is actually an affine open subset of $X_\Gamma$. 
Hence,
 $\acn(X_\Gamma)$ is the smallest non-negative integer $j$, for which there exist {\em cusp forms} $F_0,\dots, F_j$ for $\Gamma$ having no common zero on $X_\Gamma$.
%of some positive weight $k$ 
%such that 
%\[
%\dv(F_0)\cap\dots \cap \dv(F_j)= \emptyset
%\]

It is our goal to find upper bounds for $\acn(X_{\Gamma_n})$, where $\Gamma_n=\Symp_n(\Z)$. 
According to \cite[Theorem 4]{At}, we have $\acn(X_{\Gamma_n})\geq n(n-1)/2$. 
It is 
%believed that this lower bound is actually an equality.
an interesting question whether this lower bound is actually an equality. 
If $n=1$, then $\acn(X_{\Gamma_1})= 0$, since $X_{\Gamma_1}$ is affine. For $n=2$ it is easy to see that 
\[
\acn(X_{\Gamma_2})= 1,
\]
since the Igusa cusp forms $\chi_{10}$ and $\chi_{12}$ for $\Gamma_2$ (see p.~848 in \cite{Ig2}) have no common zeros on $\H_2$. 
However, for general $n$ not much is known in this direction. It even seems to be difficult to find upper bounds that improve upon the trivial bound $n(n+1)/2$.

Here we use theta functions to obtain upper bounds for $\acn(X_{\Gamma_n})$ 
for further small values of $n$, see Propositions \ref{prop:acn3} and \ref{prop:acn4}.
Recall that for every theta characteristic 
$m=\left(\begin{smallmatrix}a\\b \end{smallmatrix}\right)\in \Z^{2n}$ of genus $n$ and for $\tau\in \H_n$ 
there is a theta constant defined by
\[
\theta\begin{bmatrix} m\end{bmatrix} (\tau) = \theta\begin{bmatrix} a \\ b\end
{bmatrix} (\tau) = \sum_{x\in \Z^n} 
\exp\left( \pi i \left( {}^t(x+a/2)\tau (x+a/2) + {}^t(x+a/2)b\right)\right).
\]
Since this function depends up to the sign only on $m$ modulo $(2\Z)^{2n}$, it 
is common to define 
$\theta\begin{bmatrix} m\end{bmatrix} (\tau) = \theta\begin{bmatrix} \iota(m) 
\end{bmatrix} (\tau)$ for  $m\in \F_2^{2n}$.
Here $\iota$ is defined by the embedding 
\[
\F_2\to \Z, \quad 0\mapsto 0,\; 1\mapsto 1.
\]
A theta
characteristic $m=\left(\begin{smallmatrix}a\\b \end{smallmatrix}\right)\in \F_2^{2n}$ is called even if ${}^ta b \equiv 0\pmod{2}$, and odd otherwise. There are $2^{n-1}(2^n+1)$ even and $2^{n-1}(2^n-1)$ odd theta characteristics in genus $n$. The function $\theta\begin{bmatrix} m\end{bmatrix}$ vanishes identically if and only if $m$ is odd. 
We denote by $m\mapsto \gamma.m$ the usual action of $\gamma\in \Gamma_n$ on the theta characteristcs, see e.g.~\cite[Chapter I.3]{Fr}. There are two orbits under this action, given by the even and the odd characteristics. We write $\calE_n\subset \F_2^{2n}$ for the set of even theta characteristics. 
The theta transformation formula implies that 
\[
\theta^8[\gamma.m](\gamma\tau) = \det(c\tau + d)^4 \cdot \theta^8[m](\tau)
\]
for all $\gamma=\kabcd\in \Gamma_n$.
In particular, the theta constants are modular forms for a 
congruence subgroup of $\Gamma_n$. Sightly more precisely, their eighth powers are 
modular forms of weight $4$ for the principal congruence subgroup $\Gamma_n(2)
\subset \Gamma_n$ of level two, see e.g.~\cite[Chapter V.1]{Ig}.
For a subset $\calS\subset \calE_n$ we write 
\[
\theta[\calS](\tau) = \prod_{m\in \calS} \theta[m](\tau)
\]
for the corresponding product of theta constants. It is a modular form of weight $\# \calS/2$.
% for a congruence subgroup of $\Gamma_n$. 

%Sometimes we decompose theta characteristics as follows. 
If $n=n_1+n_2$ and $m_i=\left(\begin{smallmatrix}a_i\\b_i \end{smallmatrix}\right)\in \F_2^{2n_i}$ are theta characteristics of genus $n_1$ and $n_2$, respectively, then 
\begin{align}
\label{eq:chardec}
a=\begin{pmatrix}a_{1}\\ a_2\end{pmatrix},\quad 
b=\begin{pmatrix}b_{1}\\  b_2\end{pmatrix}
\end{align}
determine a theta characteristic $m=\left(\begin{smallmatrix}a\\b \end{smallmatrix}\right)\in \F_2^{2n}$. In this case we write $m=(m_1,m_2)$. If $\calS_1\subset\calE_{n_1}$ and  $\calS_2\subset\calE_{n_2}$ are subsets, then $\calS_1\times \calS_2\subset \calE_n$ and we may consider the corresponding theta products $\theta[\calS_1\times \calS_2]$ 
%of weight $\# \calS_1\# \calS_2/2$ 
of genus $n$. The following lemma is an easy consequence of the behavior of theta constants under the Siegel phi operator, see e.g.~Remark 3.10 in Chapter I.3 of \cite{Fr}.

\begin{lemma}\label{lem:cusp}
Let $\calS\subset \calE_n$ and assume that $\# \calS> 2^{n-1}(2^{n-1}+1)$. Then $\theta[\calS]$ is a cusp form.
\end{lemma}

\begin{comment}
\begin{proof}
Let $m=\left(\begin{smallmatrix}a\\b \end{smallmatrix}\right)\in \F_2^{2n}$ be an even theta characteristic and write $a={}^t(a_1,\dots,a_n)$ and  $b={}^t(b_1,\dots,b_n)$.  
Recall from Remark 3.10 in Chapter I.3 of \cite{Fr} that $\theta[m]$ vanishes at the standard boundary components of degree $j$ for $0\leq j <n$ if and only if $a_n$ is odd.
The number of $m\in \calE_n$ with $a_n$ odd is easily seen to be equal to $2^{2(n-1)}$.
Hence any $\calS\subset \calE_n$ with 
\[
\#\calS > \# \calE_n - 2^{2(n-1)} = 2^{n-1}(2^{n-1}+1)
\]
contains at least one such $m$. Then $\theta[\calS]$ vanishes at the standard boundary components of degree $j$ for $0\leq j <n$.

For $\gamma\in \Gamma_n$ we have 
\begin{align}
\label{eq:thetaact}
\theta[\calS]^8\mid \gamma = \theta[\gamma^{-1}\calS]^8.
\end{align}
Since $\gamma^{-1}\calS$ is a subset of $\calE_n$ of the same cardinality as $\calS$, the same argument shows that $\theta[\calS]^8\mid \gamma$ vanishes at the standard boundary components. Consequently $\theta[\calS]$ vanishes at all rational boundary components.
\end{proof}
\end{comment}

As in \cite{FP} and \cite{Ig2} we consider the following modular forms of genus $n$:
\begin{align}
\label{eq:Fnull}
F_{\text{null}} &=\theta[\calE_n],\\
\label{eq:F_1}
F_1 &= \sum_{m\in \calE_n} \theta[\calE_n\setminus\{m\}]^8. 
\end{align}
These forms have weight $2^{n-2}(2^n+1)$ and $2^{n+1}(2^n+1)-4$, respectively. If $n\geq 2$, then according to Lemma \ref{lem:cusp}, they are both cusp forms for $\Gamma_n$. 
If $n=1$, then $F_{\text{null}}^8=\Delta$. For $n=2$ we have have $F_{\text{null}}^2=\chi_{10}$. 

We define the push-forward $M_k(\Gamma_n(2))\to M_k(\Gamma_n)$ of modular forms of genus $n$ by 
\begin{align}
P_n(f) = \sum_{\gamma\in \Gamma_n(2)\bs \Gamma_n} f \mid_{k} \gamma
\end{align}
for $f\in M_k(\Gamma_n(2))$. It takes cusp forms to cusp forms.
For example, let $\calE_n^*=\calE_n\setminus\{0\}$. Then the theta transformation formula 
%\eqref{eq:thetaact} 
and the fact that $ \Gamma_n(2)\bs \Gamma_n$ acts transitively on $\calE_n$ imply that 
\[
P_n(\theta[\calE_n^*]^8)= C\cdot F_1,
\]
where $C$ denotes the order of the stabilizer in $ \Gamma_n(2)\bs \Gamma_n$ of the zero characteristic.

For $n=3$ the form $F_{\text{null}}$ is a cusp form of weight $18$, denoted $\chi_{18}$ in \cite{Ig2}, and $F_1$ is a cusp form of weight $140$, denoted $\Sigma_{140}$ in \cite{Ig2}.
By a result of Igusa \cite[Lemma 11]{Ig2}, the common vanishing locus $\dv(F_{\text{null}})\cap \dv(F_1)$ is exactly the reducible locus of $X_{\Gamma_3}$, i.e., the image of the natural map
$X_{\Gamma_1}\times X_{\Gamma_2}\to X_{\Gamma_3}$.

We now construct two cusp forms for $\Gamma_3$ whose simultaneous vanishing locus is disjoint from the reducible locus. To this end we consider the subsets 
\begin{align*}
\calE_{1,2}&:= \calE_1\times \calE_2,\\
\calE_{1,1,1}&:= \calE_1\times \calE_1\times \calE_1
\end{align*}
%It has $30$ elements. 
of $\calE_3$.
By Lemma \ref{lem:cusp}, 
the corresponding theta products  $\theta[\calE_{1,2}]^8$ and $\theta[\calE_{1,1,1}]^8$
are cusp forms  for $\Gamma_3(2)$ of weight $120$ and $108$, respectively. 
We define cusp forms for $\Gamma_3$ by
\begin{align*}
%\label{eq:F12}
F_{1,2}&= P_3( \theta[\calE_{1,2}]^8),\\
%\label{eq:F111}
F_{1,1,1}&= P_3(\theta[\calE_{1,1,1}]^8 ).
\end{align*}

\begin{lemma}
\label{lem:F12}
i) 
The restriction of $F_{1,2}$ to $X_{\Gamma_1}\times X_{\Gamma_2}$ is given by 
\[
F_{1,2}\zxz{\tau_1}{0}{0}{\tau_2 } = C\cdot \Delta(\tau_1)^{10}\cdot \chi_{10}(\tau_2)^{12},
\]
where $\tau_1\in \H_1$, $\tau_2\in \H_2$, and $C$ is the order of the stabilizer in $ \Gamma_3(2)\bs \Gamma_3$ of the set $\calE_{1,2}$.

ii) The restriction of $F_{1,1,1}$ to $X_{\Gamma_1}^3$ is given by 
\[
F_{1,1,1}\begin{pmatrix}\tau_1 & &\\
& \tau_2 &\\
&& \tau_3\end{pmatrix}= C\cdot \Delta(\tau_1)^{9} \Delta(\tau_2)^{9} \Delta(\tau_3)^{9},
\]
where $C$ is the order of the stabilizer in $ \Gamma_3(2)\bs \Gamma_3$ of the set $\calE_{1,1,1}$.
\end{lemma}

\begin{proof}
We only carry out the proof of (i), since the proof of (ii) is analogous.
We compute the restriction to $\H_1\times \H_2$ of the summands 
\[
\theta[\calE_{1,2}]^8 \mid_{120} \gamma = \theta[\gamma^{-1}\calE_{1,2}]^8 
\] 
in the definition of $P_3(\theta[\calE_{1,2}]^8)$ 
for every $\gamma\in \Gamma_3(2)\bs \Gamma_3$.
There are two cases.

First, if $\gamma$ takes $\calE_{1,2}$ to itself, then $\theta[\gamma^{-1}\calE_{1,2}]^8= \theta[\calE_{1,2}]^8$, and therefore the restriction is given by
%By construction, its restriction to 
%$\H_1\times \H_2$ is given by
\begin{align*}
%\label{eq:res}
\theta[\calE_{1,2}]^8 \zxz{\tau_1}{0}{0}{\tau_2 }&= \prod_{\substack{m_1\in \calE_1\\ m_2\in \calE_2}} \theta[m_1](\tau_1)^8 \theta[m_2](\tau_2)^8\\
\nonumber
&= \theta[\calE_1](\tau_1)^{80}\cdot \theta[\calE_2](\tau_2)^{24} = \Delta(\tau_1)^{10}\cdot \chi_{10}(\tau_2)^{12}.
\end{align*}

Second, if $\gamma$ does not take $\calE_{1,2}$ to itself, then there exists an $m\in \calE_{1,2}$ with 
$\tilde m:=\gamma^{-1} m \notin \calE_{1,2}$. Writing $\tilde m=(\tilde m_1,\tilde m_2) $ as in \eqref{eq:chardec}, we see that $\tilde m_1\in \F_2^2$ must be odd. But this implies that 
\[
\theta[\tilde m]^8 \zxz{\tau_1}{0}{0}{\tau_2 }= \theta[\tilde m_1](\tau_1)^8 \theta[\tilde m_2](\tau_2)^8 =0.
\]
This proves the claim.
\end{proof}

\begin{proposition}
\label{prop:acn3}
The cusp forms $F_{\text{null}}$, $F_1$, $F_{1,2}$, and $F_{1,1,1}$ for $\Gamma_3$ have no common zero on $\H_3$. In particular, $\acn(X_{\Gamma_3})=3$. 
\end{proposition}

\begin{proof}
According to \cite[Lemma 11]{Ig2}, the common vanishing locus of $F_{\text{null}}$ and $F_1$ is exactly the reducible locus of $X_{\Gamma_3}$.
%, i.e., the image of the natural map
%\[
%\varphi:
%X_{\Gamma_1}\times X_{\Gamma_2}\to X_{\Gamma_3}.
%\]
But Lemma \ref{lem:F12} implies that $F_{1,2}$ and $F_{1,1,1}$ never vanish simultaneously on the reducible locus. 
\end{proof}

We now turn to the case $n=4$. Following \cite{FP} we consider the following modular forms for $\Gamma_4$:
\begin{align}
\label{eq:FH}
F_{H} &=\sum_{A\subset \F_2^{2n}} \theta[\calE_n\setminus A]^8,\\[-1ex]
\label{eq:FT}
F_T &= 2^n\sum_{m\in \calE_n} \theta[m]^{16} - \bigg(\sum_{m\in \calE_n} \theta[m]^8\bigg)^2 . 
\end{align}
In \eqref{eq:FH}, the subsets $A$ runs through all suitable sets of $v(n)$ characteristics corresponding to the irreducible components of the hyperelliptic locus of $X_{\Gamma_n(2)}$ as in %\cite[Theorem 2]{SM} or 
\cite[Theorem 2]{FP}.
These forms have weight $2\binom{2n+2}{n+1}$ and $8$, respectively. 
For $n=4$, we have $v(4)=10$ and $F_H$ is cuspidal by Lemma \ref{lem:cusp}. Moreover, $F_T$ is the Schottky cusp form.
%By Remark \ref{rem:cusp} they are both cusp forms.

\begin{lemma}
\label{lem:4red}
For $n=4$, the common vanishing locus of the cusp forms $F_{\text{null}}$, $F_1$, $F_H$, $F_T$ is the reducible locus $X^{\mathrm{red}}_{\Gamma_4}$ of  $X_{\Gamma_4}$. 
\end{lemma}

\begin{proof}
By a result of Igusa \cite[p.~544]{Ig3}, the intersection $\dv(F_{\text{null}})\cap \dv(F_1)\cap \dv(F_T)$ is equal to the union of the hyperelliptic locus and the reducible locus  of  $X_{\Gamma_4}$. According to \cite[Lemma 2]{FP}, the form $F_H$ never vanishes on the hyperelliptic locus, while it vanishes on the reducible locus by the argument in the proof of  \cite[Theorem 1]{FP}.
\end{proof}

\begin{proposition}
\label{prop:acn4}
There are 5 cusp forms for $\Gamma_4$ that have no common zero on  $X^{\mathrm{red}}_{\Gamma_4}$. 
In particular, $\acn(X_{\Gamma_4})\leq 8$. 
\end{proposition}

\begin{proof}
We argue similarly as in the proof of Proposition \ref{prop:acn3}. We use the pushforward of suitable theta products in genus $4$ to construct the desired cusp forms. We let
\begin{align*}
F_{1,3} &= P_4(\theta[\calE_1\times \calE_3]^8),\\
F_{1,3}^* &= P_4(\theta[\calE_1\times \calE_3^*]^8),\\
F_{1,1,2} &= P_4(\theta[\calE_1\times\calE_1\times  \calE_2]^8),\\
F_{1,1,1,1} &= P_4(\theta[\calE_1\times \calE_1\times \calE_1\times \calE_1]^8).
\end{align*}
By Lemma \ref{lem:cusp} these modular forms are cuspidal. Their weights are 
$432$, $420$, $360$, $324$, respectively. Their restriction to $\H_1\times \H_3\subset \H_4$ can be computed as in Lemma \ref{lem:F12}. For instance, we have
\begin{align*}
F_{1,3} \zxz{\tau_1}{0}{0}{\tau_2} &= C_1 \cdot \Delta(\tau_1)^{36} F_{\text{null}}(\tau_2)^{24}, \\
F_{1,3}^* \zxz{\tau_1}{0}{0}{\tau_2} &= C_2 \cdot \Delta(\tau_1)^{35} F_{1}(\tau_2)^{3},
\end{align*}
where $C_1,C_2$ are suitable positive integral constants and $\tau_1\in \H_1$, $\tau_2\in \H_3$. By means of Proposition \ref{prop:acn3} it can be concluded that these four cusp forms never vanish simultaneously on the image of the natural map $X_{\Gamma_1}\times X_{\Gamma_3}\to X_{\Gamma_4}$.

If we define in addition the cusp form 
\begin{align*}
F_{2,2} &= P_4(\theta[\calE_2\times \calE_2]^8)
\end{align*}
of weight $400$, then a similar argument shows that the three cusp forms $F_{2,2}$, $F_{1,1,2}$, $F_{1,1,1,1}$ never vanish simultaneously on the image of the natural map $X_{\Gamma_2}\times X_{\Gamma_2}\to X_{\Gamma_4}$. Consequently, the five cusp forms $F_{1,3}$, $F_{1,3}^*$, $F_{1,1,2}$, $F_{1,1,1,1}$, $F_{2,2}$ have no common zero on  $X^{\mathrm{red}}_{\Gamma_4}$. 

The bound on $\acn(X_{\Gamma_4})$ now follows by means of Lemma \ref{lem:4red}.
\end{proof}

\begin{comment}

In this context, we also mention the following result of Freitag and Salvati Manni \cite{FM}.

\begin{theorem}
\label{thm:FSM}
Consider the set of all characteristics $\left(\begin{smallmatrix}a\\b \end{smallmatrix}\right)\in (\Z/2\Z)^{2n}$ for which $b=0$ and $a={}^t(a_1,\dots, a_n)$ with $a_n=0$. The corresponding $2^{n-1}$ theta constants $\theta\left[\begin{smallmatrix} a \\ b\end{smallmatrix}\right]$ have no common zero on $\H_n$. 
\end{theorem}

Unfortunately, this cannot be used in a direct way to construct affine open coverings of $X_{\Gamma_n(2)}$, since the individual theta constants are not cuspidal.
\end{comment}
%
%\begin{corollary}
%\label{cor:acnbound}
%Let $\Gamma\subset \Gamma_n(2)$ be a congruence subgroup. Then 
%\[
%\acn(X_\Gamma) \leq \min(2^{n-1}-1,\, n(n+1)/2).
%\]
%\end{corollary}
%
%\begin{proof}[Proof of Corollary \ref{cor:acnbound}]
%The fourth powers $\theta^4\left[\begin{smallmatrix} a \\ 0\end{smallmatrix}\right]$ of the $2^{n-1}$ theta constants of Theorem~\ref{thm:FSM} are modular forms of weight $2$ for $\Gamma$ without any common zero on $\H_n$. Hence the complements of their vanishing loci definine an affine open covering of $X_\Gamma$. This shows that $\acn(X_\Gamma) \leq 2^{n-1}-1$. Combing this with the trivial bound, we obtain the assertion.
%\end{proof}

\subsection{A special case of Raynaud's Lefschetz Theorem}

\label{sect:4.2}

Here we state a special case of an algebraization theorem of Raynaud. We will use it to show that the map \eqref{eq:keymap} is surjective provided that a certain upper bound for the affine covering number holds.

\begin{theorem}
\label{thm:alg}
Let $k$ be a field and let $X$ be a scheme over $k$ which is of finite type, separated, and proper. 
%Let $k$ be a field, $X$ be a proper algebraic variety over $k$, and  
Let $Y\subset X$ be a closed subscheme, and write $\hat X$ for the formal completion of $X$ along $Y$.
Assume that the complement $U =X \setminus Y$ is smooth and satisfies 
\[
\acn(U)\leq \dim(U)-2.
\]
Then for every finite locally free sheaf  $\calF$ on $X$ the canonical morphism 
\[
H^0(X,\calF) \to H^0(\hat X,\hat \calF)
\]
is an isomorphism.
\end{theorem}

\begin{proof}
This is a special case of \cite[Corollaire 2.8]{Ray}.  To see this, we note that the assumptions imply that 
$X$ has a dualizing complex. 
%endow $Y$ with the reduced scheme structure. 
We put $c=\acn(U)$ and let $d$ be an integer such that for all points $u\in U$ we have 
\begin{align}
\label{eq:cond}
\operatorname{depth}_u(\calF) \geq \inf(d-c, d-\operatorname{trdeg} (\kappa(u)/k)),
\end{align}
where $\kappa(u)$ denotes the residue field of $u$. 
Then \cite[Corollaire 2.8]{Ray} states that for all $i<d-c-1$ the canonical morphism 
\[
H^i(X,\calF) \to H^i(\hat X,\hat \calF)
\]
is an isomorphism.

To deduce the claimed result, we put $i=0$ and $d=c+2$. Since $U$ is smooth and $\calF$ is locally free, we have 
\[
\operatorname{depth}_u(\calF) 
%= \operatorname{depth}_{\calO_{X,u}}(\calF_u)= \dim \calO_{X,u}
= \dim(U)- \operatorname{trdeg} (\kappa(u)/k).
\]
Hence, condition \eqref{eq:cond} simplifies to
\begin{align*}
\dim(U)\geq \inf(2+ \operatorname{trdeg} (\kappa(u)/k), c+2).
\end{align*}
Our assumption $c=\acn(U)\leq \dim(U)-2$ implies that this condition is satisfied. 
\end{proof}

\begin{remark} 
It is an interesting question whether the hypothesis of the theorem can be weakened. For instance, can the upper bound on $\acn(U)$ be replaced by an upper bound on the affine stratification number (see \cite{RV}) or on the cohomological dimension of $U$?
\end{remark}

We now specialize to the case that the ground field is $\C$. By means of GAGA we also have an analogous result for complex analytic spaces.
If $X$ is a scheme of finite type over $\C$,  we denote by $X^{\mathrm{an}}$ its analytification, and let $h:X^{\mathrm{an}}\to X$ be the analytification morphism, see \cite[Expos\'e XII]{SGA1}. If $\calF$ is a coherent sheaf  on $X$, we write $\calF^{\mathrm{an}}= h^*(\calF)$ for its analytification. 

%$X\to \Spec k$ be a morphism of schemes which is proper, separated, and of finite type. 

\begin{corollary}
\label{cor:an}
Let $X$ be a scheme over $\C$ which is of finite type, separated, and proper. 
%Let $k$ be a field, $X$ be a proper algebraic variety over $k$, and  
Let $Y\subset X$ be a closed subscheme, and write $\hat X^{\mathrm{an}}$ for the formal completion of $X^{\mathrm{an}}$ along $Y^{\mathrm{an}}$ (in the category of formal complex analytic spaces as in \cite{Ba}).
Assume that the complement $U =X \setminus Y$ is smooth and satisfies 
\[
\acn(U)\leq \dim(U)-2.
\]
Then for every finite locally free $\calO_{X^{\mathrm{an}}}$-module  $\calG$ on $X^{\mathrm{an}}$ the canonical morphism 
\[
H^0(X^{\mathrm{an}},\calG) \to H^0(\hat X^{\mathrm{an}},\hat \calG)
\]
is an isomorphism. Here $\hat \calG$ denotes the completion of $\calG$ along  $Y^{\mathrm{an}}$ in the category of formal complex analytic spaces.
\end{corollary}

\begin{proof}
By the GAGA theorem, there exists a unique  finite locally free $\calO_{X}$-module $\calF$ on $X$ such that $\calG= \calF^{\mathrm{an}}$. By Theorem \ref{thm:alg} 
the canonical homomorphism 
\begin{align*}
H^0(X,\calF) \to H^0(\hat X,\hat \calF)
\end{align*}
is an isomorphism. The left hand side is canonically isomorphic to $H^0(X^{\mathrm{an}},\calF^{\mathrm{an}})$.
Let $\calI\subset\calO_X$ be the ideal sheaf defining $Y$.
%, and let $\calY_k$ be the $k$-th infinitesimal neighborhood of $Y$. 
The canonical homomorphism 
\begin{align}
\label{eq:compl}
H^0(\hat X,\hat \calF)\to \lim_k H^0(X, \calF/\calI^k \calF)
\end{align}
is an isomorphism.
%See e.g. \cite[(4.1.3)]{EGA} and \cite[Theorem 4.1.5]{EGA} for $f=\id$.
By GAGA for the coherent sheafs $\calF/\calI^k \calF$, the right hand side of \eqref{eq:compl} is isomorphic to 
\[
\lim_k H^0(X^{\mathrm{an}}, (\calF/\calI^k \calF)^{\mathrm{an}}) \cong H^0(\hat X^{\mathrm{an}},(\calF^{\mathrm{an}})\hat{\,}\,).
\]
Putting these maps together, we obtain a natural isomorphism
\[
H^0(X^{\mathrm{an}},\calF^{\mathrm{an}})\to H^0(\hat X^{\mathrm{an}},(\calF^{\mathrm{an}})\hat{\,}\,).
\]
This completes the proof of the corollary.
\end{proof}

\subsection{Algebraization of formal Siegel modular forms}

Here we combine Corollary \ref{cor:an} and Theorem \ref{prop:ffj} to derive a modularity result for formal Siegel modular forms of cogenus $1$. We use the notation of Section \ref{sect:3}. In particular, $\Gamma\subset \Symp_n(\Q)$ denotes an arithmetic subgroup. 
%Throughout we assume that the group $\Gamma$ satisfies Assumption~\ref{ass:gen}.

\begin{theorem}
\label{thm:maingen}
Assume that $\acn(X_\Gamma)\leq \frac{n(n+1)}{2} -2$. Then the natural map 
\begin{align}
\label{eq:fjmap}
H^0(X^*_\Gamma, \omega^{\otimes k})\to  \operatorname{FM}^{(n,1)}_k(\Gamma)
\end{align}
taking a modular form to its cogenus $1$ formal Fourier-Jacobi expansions is an isomorphism. 
\end{theorem}

\begin{proof}
We first assume that $\Gamma=\Gamma_n(N)$ is the principal congruence subgroup of level $N\geq 3$ and that the condition $\acn(X_\Gamma)\leq \frac{n(n+1)}{2} -2$ holds. Then $\Gamma$ acts freely on $\H_n$ and $\omega^k$ is locally free of rank $1$. 
%Moreover, Assumption~\ref{ass:gen} is satisfied.
%
The map \eqref{eq:fjmap} is part of the commutative diagram
\[
\xymatrix{
 H^0(X^*_\Gamma, \omega^{\otimes k}) \ar[r] \ar[d] & \operatorname{FM}^{(n,1)}_k(\Gamma)\\
 H^0(\hat X^*_\Gamma, \hat\omega^{\otimes k})\ar[ur] &},
\]
where the diagonal arrow is given by \eqref{eq:mapffj}. According to Theorem \ref{prop:ffj} this diagonal map is an isomorphism. Moreover, by Corollary \ref{cor:an} the vertical map is an isomorphism. This implies the assertion.

Now we consider the case that $\Gamma$ is an arbitrary arithmetic subgroup of $\Symp_n(\Q)$.
We choose a principal congruence subgroup $\Gamma'$ of level $N\geq 3$ such that  $\Gamma'\subset \Gamma$. Since finite morphisms are affine, we have 
$\acn(X_{\Gamma'})\leq \acn(X_{\Gamma})$.
The covering map $\pi: X^*_{\Gamma'}\to  X^*_{\Gamma}$ induces a pull back map 
\begin{align}
\label{eq:pbf}
\pi^*: \operatorname{FM}^{(n,1)}_k(\Gamma)\to \operatorname{FM}^{(n,1)}_k(\Gamma').
\end{align}
%and also $H^0(\hat X^*_\Gamma, \hat\omega^{\otimes k}) \to H^0(\hat X^*_{\Gamma'}, \hat\omega^{\otimes k})$. 
Hence, if $f=(f_F)_{F\in I_{n-1}}\in  \operatorname{FM}^{(n,1)}_k(\Gamma)$, 
we may apply 
the above argument to deduce that $\pi^*(f)\in\operatorname{FM}^{(n,1)}_k(\Gamma')$ is the image of an element $g\in  H^0(X^*_{\Gamma'}, \omega^{\otimes k})$. 
Then at each boundary component $F\in I_{n-1}$ the series $f_F$ is the (convergent!) Fourier-Jacobi expansion as in \eqref{eq:fsgen} of the 
holomorphic modular form $g$ for $\Gamma'$. But now condition (i) of 
Definition \ref{def:symffj} implies that 
\[
g\mid_k \gamma = f_F\mid_k \gamma = f_{\gamma^{-1} F} = g
\]
for every $\gamma \in \Gamma$. Hence $g\in H^0(X^*_\Gamma, \omega^{\otimes k})$, concluding the proof of the corollary.
\end{proof}

\begin{corollary}
\label{cor:maingen}
Assume that $2\leq n \leq 4$. Then the natural map \eqref{eq:fjmap}
%\[
%H^0(X^*_\Gamma, \omega^{\otimes k})\to  \operatorname{FM}^{(n,1)}_k(\Gamma)
%\]
%taking a modular form to its cogenus $1$ formal Fourier-Jacobi expansion 
is an isomorphism. 
\end{corollary}

\begin{proof}
If $\Gamma$ is contained in the full Siegel modular group $\Gamma_n$, the assertion follows from  Theorem \ref{thm:maingen} combined with the bounds on $\acn(X_{\Gamma_n})$ of Section \ref{sect:4.1}.

Otherwise, we consider the auxiliary congruence subgroup $\Gamma'= \Gamma\cap \Gamma_n$. The covering map $\pi: X^*_{\Gamma'}\to  X^*_{\Gamma}$ induces a pull back  map as in \eqref{eq:pbf}. Pulling back $f\in \operatorname{FM}^{(n,1)}_k(\Gamma)$,  
we find that $\pi^*(f)$ is the image of an element $g\in  H^0(X^*_{\Gamma'}, \omega^{\otimes k})$.  As in the proof of Theorem~\ref{thm:maingen} we conclude that $g$ actually descends to an element of $H^0(X^*_{\Gamma}, \omega^{\otimes k})$.
\end{proof}

\begin{remark}
i) Theorem \ref{thm:maingen} and Corollary \ref{cor:maingen} have natural generalizations to vector valued modular forms transforming with a finite dimensional representation of $\Gamma$. Alternatively, one can deduce such results for vector valued forms from those for scalar forms by means of the argument of \cite{Br1}.

ii) Using induction on the cogenus as in \cite[Lemma 5.2]{BR} one can also deduce an analogue for formal Siegel modular forms of higher cogenus $l<n$.

iii) The above results are also valid in the slighltly more general case when $\Gamma$ is an arithmetic subgroup of $\Symp_n(\R)$ (rather than of $\Symp_n(\Q)$), that is, a subgroup which is commensurable with $\Gamma_n$. Note that by \cite{Ch} the elements of such a subgroup are automatically projective rational. 
\end{remark}

\subsection{The case of the paramodular group}
\label{sect:4.4}

Here we apply Theorem \ref{thm:maingen} to prove the modularity of certain formal Fourier-Jacobi series for the paramodular group of genus $2$.

Throughout this subsection we let $n=2$ and let $N$ be a positive integer. Recall that the paramodular group of level $N$ is the arithmetic subgroup $K(N)\subset \Symp_2(\Q)$ consisting of  matrices of the form 
\begin{align}
\label{eq:KN}
\begin{pmatrix}
* & N* & * & *\\
* & * & * & */N\\
* & N* & * & *\\
N* & N* & N* & *
\end{pmatrix},
\end{align}
where the stars stand for integral entries.
For every exact divisor $d||N$ we put $d'=N/d$ and choose $\alpha,\beta,\gamma,\delta\in \Z$ such that 
$
\alpha\delta d-\beta\gamma d'=1
$. 
Then the matrix 
\[
V_d = \frac{1}{\sqrt{d}}
\begin{pmatrix}
d\delta & -N\gamma & 0 & 0\\
-\beta & d\alpha & 0 & 0\\
0 & 0 & d\alpha  & \beta \\
0 & 0 & N \gamma  & d\delta 
\end{pmatrix}
\]
belongs to $\Symp_2(\R)$ and is projective rational. The coset $V_d K(N)$ is independent of the choices of the parameters. Moreover, 
we have $ V_d^2\in K(N)$ and $V_d K(N) V_d= K(N)$. 
In the special case where $d=N$, we may choose $\beta=1$, $\gamma=-1$, and $\alpha=\delta=0$. Then we obtain for $V_N$ the (projective) involution 
\[
\mu_N = \frac{1}{\sqrt{N}}
\begin{pmatrix}
0 & N & 0 & 0\\
-1 & 0 & 0 & 0\\
0 & 0 & 0 & 1\\
0 & 0 & -N & 0
\end{pmatrix}.
\]

We denote by $K(N)^*$ the arithmetic subgroup of $\Symp_2(\R)$ which is generated by $K(N)$ and the $V_d$ for $d||N$. It contains $K(N)$ as a normal sugbroup, and 
\[
K(N)^*/K(N) \cong (\Z/2\Z)^{\nu(N)},
\]
where $\nu(N)$ denotes the number of positive exact divisors of $N$, see e.g.~\cite{GH}.

The structure of the rational boundary components of $K(N)$ and $K(N)^*$ is known, see e.g.~\cite{PY}. If $N$ is square-free, the situation is particularly simple. 
Then for $K(N)^*$ there is exactly one orbit of $0$-dimensional rational boundary components and one orbit of $1$-dimensional rational boundary components. 
For the group $K(N)$ there is one orbit of $0$-dimensional rational boundary components, and there are $\nu(N)$ orbits on $1$-dimensional rational boundary components. The orbits can be represented by the $V_d F_1$ for $d|| N$, and these representatives intersect exactly in the standard $0$-dimensional boundary component $F_0$. Moreover, we have 
$K(N)^*_{F_1}= K(N)_{F_1}$, and $K(N)^*_{F_0}$ is obtained from $K(N)_{F_0}$ by extending with the involutions $V_d$.

%The group $K(N)^*$ is an example of such a group.

Let $f$ be a formal Fourier-Jacobi series of weight $k$ for the standard boundary component $F_1$ and the group $K(N)$. Denote by 
\begin{align*}
f(\tau) = \sum_{T\in \Sym_{2}(\Q)} a(T)  \,e(\tr T \tau)
\end{align*} 
its formal Fourier expansion at the boundary component $F_0$ as in \eqref{eq:fouf}.
We assume that $f$ is {\em strongly symmetric} in the following strong sense: There exists a character $\chi_f: K(N)^*/K(N)\to \{\pm 1\}$ such that 
\begin{align}
\label{eq:strongsymm0}
f\mid_k \gamma = \chi_f(\gamma) f
\end{align} 
for all $\gamma\in K(N)^*_{F_0}$. This condition can be rephrased as 
\begin{align}
\label{eq:strongsymm}
a(T[{}^t u]) = \chi_f\zxz{{}^tu^{-1}}{0}{0}{u}\det(u)^k a(T) 
\end{align} 
for all $T\in \Sym_2(\Z)^\vee$ positive semi-definite with $N\mid T_2$, and for all $u\in \Gamma_0(N)^*\subset \Sl_2(\R)$. Here the latter group denotes the extension of $\Gamma_0(N)$ by all Atkin-Lehner involutions $W_d$ with $d||N$. We obtain the following result.

\begin{theorem}
\label{thm:KN}
Let $N$ be a square-free positive integer.
Let $f$ be a strongly symmetric formal Fourier-Jacobi series of weight $k$ for the boundary component $F_1$ and the group $K(N)$. Then $f$ converges and defines the Fourier-Jacobi expansion of a 
modular form in $M_k(K(N))$.
\end{theorem}

\begin{proof}
We define a formal Siegel modular form for $K(N)$ as follows. For any rational boundary component $F\in I_1$ we choose $\delta\in K(N)^*$ such that $F=\delta^{-1} F_1$ and put 
\[
f_F := \chi_f(\delta) f\mid_k \delta.
\] 
By Remark \ref{rem:class}, $f_F$ is a formal Fourier-Jacobi series for the boundary component $F$ and the group $K(N)$. It is easily checked that the family $(f_F)_{F\in I_1}$ defines formal Siegel modular form of weight $k$ for $K(N)$.
According to Corollary \ref{cor:maingen} it must come from a holomorphic Siegel modular form in $M_k(K(N))$.
\end{proof}

\begin{remark}
It suffices to check condition \eqref{eq:strongsymm} for $u$ running through a set of generators of $\Gamma_0(N)^*$. Therefore, Theorem~\ref{thm:KN} may be useful for computations with Siegel modular forms for the paramodular group.
For instance, if $N=1,2,3$, then $\Gamma_0(N)^*$ is generated by the translation matrix $\kzxz{1}{1}{0}{1}$ and the Fricke involution. Hence, in this case the strong symmetry condition follows from the formal modularity of $f$ for $K(N)_{F_1}$ and from \eqref{eq:strongsymm0} for the {\em single} matrix $\gamma=\mu_N$. This recovers part of the main result (Theorem 1.2) of \cite{IPY}. 
\end{remark}

In \cite{IPY} the authors ask whether for general $N$ it actually suffices to require \eqref{eq:strongsymm0} for the single element $\mu_N$. More precisely, they consider a formal Fourier-Jacobi series $f$ of weight $k$ for the standard boundary component $F_1$ and the group $K(N)$ as above. They say that $f$ satisfies the {\em involution condition}, if there is an $\eps\in {\pm 1}$ such that the formal Fourier expansion of $f$ satisfies  
$f\mid_k \mu_N = \eps f$. In terms of the Fourier coefficients this means 
\begin{align}
\label{eq:weaksymm}
a\zxz{T_2/N}{-T_{12}}{-T_{12}}{N T_1} = \eps \cdot a\zxz{T_1}{T_{12}}{T_{12}}{T_2}
\end{align} 
for all $T=\kzxz{T_1}{T_{12}}{T_{12}}{T_2}\in L_{F_1}^\vee$.
It is asked in the introduction of loc.~cit.\ whether any formal Fourier-Jacobi series $f$ as before which satisfies the involution condition is the expansion of a classical holomorphic modular form of weight $k$ for the group $K(N)$?

When $N$ is prime, one could try to approach this question by using the fact that $K(N)^*$ is generated by $K(N)_{F_1}$ and $\mu_n$ to show that any formal Fourier-Jacobi series $f$  which satisfies the involution condition \eqref{eq:weaksymm} is actually strongly symmetric in the sense of  \eqref{eq:strongsymm0}. To this end one would have to show that for any $\gamma\in K(N)_{F_0}$ the formal Fourier-Jacobi series $f_{F_1}$ and $f_{\gamma F_1}$ are compatible at $F_0$.

%\texttt{It would be very interesting to see if this works! 
%For general $N$, I think one should require an involution condition for all $V_d$.}

%%%%%%%%%%%%%%%%%%%%%%%%%%%%%%%%%%%%%%%%%%%%%%%%%


\begin{thebibliography}{BHKRY}

\bibitem[Ao]{Ao} \emph{H. Aoki}, Estimating Siegel modular forms of genus $2$ using Jacobi forms, J. Math. Kyoto Univ. {\bf 40} (2000), 581--588.

\bibitem[AMRT]{AMRT} \emph{A. Ash, D. Mumfort, M. Rapoport, Y. Tai}, Smooth Compactifications of Locally Symmetric Spaces, Second Edition, 
%Cambridge Mathematical Library, 
Cambridge University Press (2010).

\bibitem[At]{At} \emph{A. Atyam}, Affine stratification of $\calA_4$, Proceedings of the American Mathematical Society {\bf 143} (2015), 4167--4175. 

%\bibitem[Bo]{Bo3}
% \emph{R. Borcherds}, The Gross-Kohnen-Zagier theorem in higher
%dimensions, Duke Math. J. \textbf{97} (1999), 219--233.
%Correction in: Duke Math J. \textbf{105} No. 1 p.183--184.


%\bibitem[BB1]{BB0} \emph{W. L. Baily and A. Borel}, On the compactification of arithmetically defined quotients of bounded symmetric domains, Bull. Amer. Math. Soc. {\bf 70} (1964), 588--593. 

\bibitem[BB]{BB} \emph{W. L. Baily and A. Borel},
Compactification of arithmetic quotients of bounded symmetric domains, 
Ann. of Math. (2) {\bf 84} (1966), 442--528. 

\bibitem[Ba]{Ba} \emph{C. Banica}, Le complete d'un espace analytique le long d'un sous-espace: un theoreme de comparison, manuscripta math. {\bf 6} (1972), 207--244.

%\bibitem[Br]{Br} \emph{J. H. Bruinier}, Borcherds products on
%  $\Orth(2,l)$ and Chern classes of Heegner divisors, Springer Lecture
%  Notes in Mathematics {\bf 1780}, Springer-Verlag (2002).

\bibitem[Br]{Br1} \emph{J. H. Bruinier}, Vector valued formal Fourier-Jacobi series, Proc. Amer. Math. Soc. {\bf 143} (2015), 505--512. 

\bibitem[BR]{BR} \emph{J. H. Bruinier and M. Raum}, Kudla's modularity conjecture and formal Fourier-Jacobi series, Forum of Mathematics, Pi {\bf 3} (2015), 30 pp. 

\bibitem[Ch]{Ch} \emph{U. Christian}, Zur Theorie der Hilbert-Siegelschen Modulfunktionen, Math. Ann. {\bf  152} (1963), 275--341.

%\bibitem[EZ]{EZ} \emph{M. Eichler and D. Zagier}, The Theory of Jacobi Forms, Progress in Math. {\bf 55}, Birkh\"auser (1985).

\bibitem[FP]{FP}  \emph{C. Fontanari and S. Pascolutti}, An affine open covering of $\mathcal{M}_g$ for $g\leq 5$, Geometriae Dedicta {\bf 158} (2012), 61--68.

\bibitem[Fr]{Fr} \emph{E. Freitag}, Siegel modular forms, Springer-Verlag (1983).

%\bibitem[FM]{FM} \emph{E. Freitag and R. Salvati Manni}, letter to the author (2021).

\bibitem[Fu]{Fu} \emph{W. Fulton}, Introduction to toric varieties, Princeton University Press (1993).

\bibitem[GH]{GH} \emph{V. Gritsenko and K. Hulek}, Minimal Siegel modular threefolds, Mathematical Proceedings of the Cambridge Philosophical Society {\bf 123} (1998), 461--485.

%\bibitem[GKZ]{GKZ} \emph{B. Gross, W. Kohnen, and D. Zagier}, Heegner
%points and derivatives of $L$-series. II.  Math. Ann.  {\bf 278}
%(1987), 497--562.

\bibitem[EGA3]{EGA} \emph{A. Grothendieck}, \'El\'ements de g\'eométrie alg\'ebrique : III. \'Etude cohomologique
des faisceaux coh\'erents, Premiere partie. Publications math\'ematiques de l'I.H.\'E.S. {\bf 11} (1961), 5--167.

\bibitem[SGA1]{SGA1} \emph{A. Grothendieck}, Rev\^etements \'etales et groupe fondamental (SGA 1), Documents Math\'ema\-tiques, Soci\'et\'e Math\'ematiques de France (2003).  

%\bibitem[Ha]{Ha} {\em R. Hartshorne}, Algebraic Geometry, Graduate Texts in Mathematics {\bf 52}, Springer-Verlag (1977).

\bibitem[HM]{HM} {\em B. Howard and K. Madapusi}, 
Kudla's modularity conjecture on integral models of orthogonal Shimura varieties, preprint (2022).
 	arXiv:2211.05108 [math.NT]

\bibitem[HKW]{HKW} {\em K. Hulek, C. Kahn, and S. Weintraub}, Moduli spaces of abelian surfaces, de Gruyter (1993).
  
\bibitem[IPY]{IPY} {\em T. Ibukiyama, C. Poor, and D. Yuen}, Jacobi forms that characterize paramodular forms,  	Abh. Math. Semin. Univ. Hambg. {\bf 83} (2013), 111--128.
%arXiv:1209.3438 [math.NT].


\bibitem[Ig1]{Ig2} {\em J. Igusa}, Modular forms and Projective Invariants, American Journal of Mathematics {\bf 89} (1967), 817--855.

\bibitem[Ig2]{Ig} {\em J. Igusa}, Theta functions, Springer-Verlag (1972).

\bibitem[Ig3]{Ig3} {\em J. Igusa}, On the irreducibility of Schottky's divisor, J. Fac. Sci. Univ. Tokyo Sect. IA Math. {\bf 28} (1982), 531--545.


\bibitem[Kr]{Kr} {\em J. Kramer}, On formal Fourier–Jacobi expansions revisited, preprint (2022).

\bibitem[Ku1]{Ku:Duke}  \emph{S. Kudla},
Algebraic cycles on Shimura varieties of orthogonal type.  Duke
Math. J.  {\bf 86}  (1997),  no. 1, 39--78.

\bibitem[Ku2]{Ku:Annals} \emph{S. Kudla},
Central derivatives of Eisenstein series and height pairings,
Ann. of Math. (2) {\bf 146} (1997), 545--646.

\bibitem[Ku3]{Ku:MSRI} {\em S. Kudla}, Special cycles and derivatives of Eisenstein series,
in {\em Heegner points and Rankin $L$-series}, Math. Sci. Res.
Inst. Publ. {\bf 49}, Cambridge University Press, Cambridge
(2004).


\bibitem[Liu]{Liu} {\em Y. Liu}, Arithmetic theta lifting and
$L$-derivatives for unitary groups, I.
Algebra and Number Theory {\bf 5:7} (2011).

%\bibitem[Mi]{Miyake} {\em T. Miyake},  Modular forms. Translated from the Japanese by Yoshitaka Maeda. Springer-Verlag, Berlin, 1989. x+335 pp.

\bibitem[Na]{Na} {\em Y. Namikawa}, Toroidal compactifications of Siegel spaces, Springer lecture notes in Mathematics {\bf 812},  Springer-Verlag (1980).

\bibitem[PY]{PY} {\em C. Proor, and D. Yuen}, The Cusp Structure of the Paramodular Groups for Degree Two, Journal of the Korean Mathematical Society {\bf 50} (2013), 445--464.

\bibitem[Ray]{Ray} \emph{M. Raynaud}, Th\'eor\`emes de Lefschetz en cohomologie coh\'erente et cohomologie \'etale, M\'emoirs de la S. M. F.  {\bf 41} (1975).

\bibitem[Rau]{Rau} \emph{M. Raum}, Formal Fourier Jacobi Expansions and Special Cycles of Codimension Two, 
Compos. Math. {\bf 151} (2015).
%arXiv:1302.0880 [math.NT].

\bibitem[RV]{RV} \emph{M. Roth and R. Vakil}, The affine stratification number and the moduli space of curves, Algebraic structures and moduli spaces, CRM Proc. Lecture Notes {\bf 38}, Amer. Math. Soc., Providence, RI (2004), 213--227.

\bibitem[Xia]{Xia} \emph{J. Xia},   Some cases of Kudla’s modularity conjecture for unitary Shimura varieties, Forum of Mathematics, Sigma {\bf 10} (2022), 37 pp.

\bibitem[Zh]{Zh} \emph{W. Zhang}, Modularity of generating functions of special cycles on Shimura varieties, Ph.D. thesis, Columbia University (2009).
\end{thebibliography}
\end{document}